\documentclass[a4paper,10pt]{amsart}
\usepackage[T1]{fontenc}
\usepackage[english]{babel}
\usepackage{amsthm, amssymb, amsfonts}
\usepackage{dsfont}
\usepackage{mathtools}
\usepackage{mathrsfs}
\usepackage[margin=1in]{geometry}
\usepackage[toc,page]{appendix}

\usepackage{biblatex}
\bibliography{references}
\usepackage{csquotes}

\usepackage{esint}

\theoremstyle{plain}
\newtheorem{theorem}{Theorem}[subsection]
\newtheorem{proposition}[theorem]{Proposition}

\newtheorem{lemma}[theorem]{Lemma}
\newtheorem{remark}[theorem]{Remark}
\theoremstyle{remark}

\theoremstyle{definition}

\DeclarePairedDelimiter{\innn}{\{}{\}}
\DeclarePairedDelimiter{\inn}{\langle}{\rangle}
\DeclarePairedDelimiter{\abs}{\lvert}{\rvert}
\DeclarePairedDelimiter{\norm}{\lVert}{\rVert}
\DeclarePairedDelimiter{\seq}{(}{)}
\newcommand{\R}{\mathbb{R}}

\newcommand{\Lip}{\mathrm{Lip}}
\newcommand{\diva}{\mathrm{div}}

\newcommand{\tr}{\mathrm{Tr}}

\newcommand{\curl}{\mathrm{curl}}

\newcommand{\molN}{\mathfrak{I}^N}
\newcommand{\mol}{\mathfrak{I}_\epsilon}
\newcommand{\hel}{\mathfrak{H}}
\newcommand{\calf}{\mathsf{f}}
\newcommand{\calX}{\mathsf{X}}
\newcommand{\calA}{\mathsf{A}}
\newcommand{\calb}{\mathsf{b}}
\newcommand{\calB}{\mathsf{B}}
\newcommand{\calF}{\mathsf{F}}
\newcommand{\calQ}{\mathsf{Q}}

\newcommand{\cof}{\mathrm{Cof}}

\newcommand{\Om}{\Omega}

\newcommand{\T}{\mathbb{T}}
\newcommand{\Z}{\mathbb{Z}}
\newcommand{\N}{\mathbb{N}}

\newcommand{\I}{\mathbb{I}}

\numberwithin{equation}{section}

\title[Local existence of smooth solutions for the semigeostrophic equations]{Local existence of smooth solutions for the semigeostrophic equations on curved domains}
\author{Lauro Silini}
\begin{document}
\maketitle
\bigskip

\begin{abstract}
We prove local-in-time existence of smooth solutions to the semigeostrophic equations in the general setting of smooth, bounded and simply connected domains of $\R^2$ endowed with an arbitrary conformally flat metric and non-vanishing Coriolis term. We present a construction taking place in Eulerian coordinates, avoiding the classical reformulation in dual variables, used in the flat case with constant Coriolis force, but lacking in this general framework.
\end{abstract}
\bigskip
\section{Introduction}
\subsection{Background and motivations}
In meteorology, the semigeostrophic equations are believed to constitute a good approximation of the atmospheric flow dynamics in the large scale setting (see for instance Cullen \cite{C06} and Hoskins \cite{H75}). The form of these equations in the flat case is
\begin{equation}\label{eq:SG_intro}
\begin{cases}
(\partial_t+u\cdot\nabla)u_G+f(u-u_G)^\perp=0,\\
u_G=\frac{1}{f}\nabla^\perp p,\\
\diva(u)=0,
\end{cases}
\end{equation}
where $u$, $f$, $p$ and $u_G$ are respectively the velocity of the fluid, the Coriolis term, the pressure and semigeostrophic wind (that is the virtual component of the velocity that takes into account the action of the large Coriolis effect against the material derivative).
The local-in-time existence of smooth solutions has been obtained by Loeper (see \cite{L06}). The global-in-time existence of weak solution has been proved by Ambrosio, Colombo, De Philippis and Figalli in \cite{ACPF12}. Those two achievements suppose the Coriolis term constant and the domain flat, since in this case it is possible to formulate a dual version of the semigeostrophic equation which is formed by a continuity equation coupled with a Monge-Ampère equation
\[
\begin{cases}
\partial_t\rho_t+\diva(\rho_t U_t)=0,\\
U_t(y)=(y-\nabla P^*(y,t))^\perp,\\
\det(D^2P^*(y,t))=\rho_t(y).
\end{cases}
\]
About this reformulation, see Benamou and Brenier \cite{BB71}.
At this stage the classical theory of optimal transport allows the construction of a solution for the dual system. For the global-in-time weak solutions, a careful application of the regularity theory for uniformly bounded Monge-Ampère equations, developed by De Philippis and Figalli in \cite{FP12},  makes the transition to the original system possible.
\bigbreak
The initial motivation of this work is the study of the incompressible semigeostrophic equations over its natural setting: a rotating sphere. At our knowledge, there is no clear generalization of the elegant machinery mentioned before to more physically accurate frameworks, even for the simple periodic case with varying Coriolis term. For this reason, a natural question is whether the passage to dual coordinates can be avoided in order to gain information on the existence in more geometrically interesting domains. This article constitutes a first result in this direction. More precisely, we study the local-in-time existence of smooth solutions. Our method, which was originally customized for the upper (and lower) hemisphere stereographically projected, is robust with respect to the geometry of the underlying domain, and generalizes to any smooth, simply connected, conformally flat domain in $\R^2$. As a remark at the end of the paper, we  will show that the method can be adapted to the non-simply connected model case of the flat torus. For what concerns the Coriolis force, we assume it to be non-vanishing on the closure of the whole domain. Considering that the Coriolis force on a rotating sphere is proportional to the height function, and hence vanishes along the equator, this is a restrictive assumption. However, as we will see heuristically in Section \ref{subsection:derivation} (or simply by looking at the definition of  $u_G$ in \eqref{eq:SG_intro}), there are evident issues of regularity where $f$ vanishes. In fact, it is not clear whether the semigeostrophic equations make any sense at all along the equator. We decided to avoid this delicate problem for now, focusing (in the case of the spherical geometry) on the existence of solutions compactly supported in the upper (or lower) hemisphere. We decided to prescribe the velocity vector field $u$ tangent to the boundary of the domain, which corresponds to the natural condition for mass preservation. The only given data at time zero is the pressure gradient $\nabla p_0$, and we suppose it satisfies a stability condition coming from the derivation of an elliptic PDE for the velocity in terms of the pressure. In the flat setting, this corresponds to the well celebrated \emph{Cullen stability principle} (see \cite[Chapter 3]{C06}), that asks $D^2p_0+\I$ to be convex.
\subsection{Derivation}\label{subsection:derivation}
We start by briefly describing the derivation of the semigeostrophic equation starting from the classical Euler equation governing the evolution of an inviscid fluid on the surface of  a rotating sphere. The generalization to conformally flat domains follows naturally after projecting everything stereographically. Consider the two dimensional sphere rotating on itself with constant velocity $\omega$. The canonical Riemannian metric is given in spherical coordinates $(\theta,\phi)$ by
\[
g=d\theta^2+\sin(\theta)d\phi^2,
\]
where $\theta\in(0,\pi)$ represents the latitude and $\phi\in (0,2\pi)$ the longitude. Denoting with $D^g$, $\diva^g$ and  $\nabla_g$ the Levi-Civita connection, the divergence and the gradient operator  associated  to $g$, and with $(\cdot)^\perp$ the counter-clockwise rotation of $\pi/2$ radians, the  evolution of a two dimensional inviscid fluid on this sphere is described by
\begin{equation}\label{eq:Euler}
\begin{cases}
(\partial_t+D^g_u)u+fu^\perp+\nabla_g p=0,\\
\diva^g(u)=0,
\end{cases}
\end{equation}
where $u$  and $p$ represent respectively the velocity and the pressure of the fluid, and $f$ is the Coriolis term which equal to $f=2\omega\cos(\theta)$. In the large scale setting it is believed that the force induced by the rotation dominates the advection term (at least far form the equator). This induces the definition of a new quantity $u_G$, called the geostrophic wind, which represents the balance
\[
fu_G^\perp+\nabla_g p=0.
\]
The semigeostrophic approximation consists into neglecting the action of the material derivative on the difference $u-u_G$ (called the  ageostrophic component of the wind), but preserving all information about the fluid velocity in the remaining terms of the equation. This means that we are asking
\[
\partial_t u_G+D^g_u u_G+fu^\perp+\nabla_g p=0.
\]
In a more compact form, we finally have the semigeostrophic equation
\begin{equation}\label{eq:primitive_SG}
\begin{cases}
(\partial_t+D^g_u)u_G+f(u-u_G)^\perp=0,\\
u_G=\frac{1}{f}\nabla^\perp_g p,\\
\diva^g(u)=0,\\
\end{cases}
\end{equation}
in its essential formulation (see \cite[Chapter 2]{C06} for the spherical case, and  \cite[Chapter 1]{C06}, \cite{H75} and \cite{L06} for the flat periodic one). Operating a stereographic projection pointed  at the South Pole,  we can see \eqref{eq:primitive_SG} as taking place in the two dimensional plane endowed with the conformal metric and Coriolis term
\begin{equation}\label{eq:spherical_coeff}
g=\frac{4}{(1+\abs{x}^2)^2}\bigl((dx^1)^2+(dx^2)^2\bigr),\quad f=2\omega\frac{1-\abs{x}^2}{1+\abs{x}^2},
\end{equation}
in canonical Cartesian coordinates $(x^1,x^2)$.\\
\par To summarise, we can give the general statement of this problem: let $\Om$ be a sufficiently smooth, bounded and simply connected domain of $\R^2$, and let $V,\varphi$ be two given smooth functions defined on $\bar\Om$. Set \begin{equation}\label{eq:general_form_metric_coriolis}
g:=e^{-2V}\bigl((dx^1)^2+(dx^2)^2\bigr),\quad\text{and  }f=e^{-\varphi},
\end{equation}
and define the endomorphism of tangent bundle
\[
J=(\cdot)^\perp:T\Om\to T\Om,\quad J=-dx^2\otimes\frac{\partial}{\partial x^1}+dx^1\otimes\frac{\partial}{\partial x^2},
\]
to be the counter-clockwise rotation of $\pi/2$-radians. Given an initial pressure gradient $\nabla^g p_0$ we ask ourselves if it is possible to find a local-in-time smooth solution of the  semigeostrophic system
\begin{equation}\label{eq:SG}
\begin{cases}
(\partial_t+D^g_u)\bigl(e^{\varphi}\nabla_g^\perp p\bigr)+e^{-\varphi}\bigl(u-e^{\varphi}\nabla^\perp_g p\bigr)^\perp=0,&\text{ in }\Om,\\
\diva^g(u)=0,&\text{ in }\Om,\\
g(u,\nu)=0,&\text{ on }\partial\Om,\\
\nabla^gp_t\mid_{t=0}=\nabla^gp_0,
\end{cases}
\end{equation}
where $\nu$ denotes the outer pointing normal vector to $\partial\Om$. In particular, when $\varphi$ and $V$ are as in \eqref{eq:spherical_coeff} and \eqref{eq:general_form_metric_coriolis}  we are in the spherical case, and when $V=\varphi=0$, we are in the classic flat case.
\subsection{Main result}For any vector field $\xi\in C^1(\Om,\R^2)$ we define the \emph{stability matrix} $\mathcal{Q}$ as
\[
\mathcal{Q}=\mathcal{Q}[D\xi,\xi]:=e^{2V+2\varphi}\Bigl(D\xi^T+(\nabla V+\nabla\varphi/2)\otimes\xi+\xi\otimes(\nabla V+\nabla\varphi/2)-\inn{\xi,\nabla V}\I\Bigr),
\]
where $\I$ denotes the identity matrix. For matrices $A$ and $B$, we will write
\[
A\geq B, \text{ whenever }\inn{A\xi,\xi}\geq\inn{B\xi,\xi}\text{ for all }\xi\in\R^2\setminus\{0\}.
\]
The main result of this paper is the following.
\begin{theorem}\label{thm_MAIN}Let $k\geq 4$ be fixed, and let $\Om$ be an open,  simply connected and bounded subset of $\R^2$ with $C^{k+1}$ boundary. Let $\nabla p_0\in H^k(\Om,\R^2)$,  and suppose that there exists $\mu_0<1$ such that
\[
\I+\mathcal{Q}[D^2p_0,\nabla p_0]\geq (1-\mu_0)\I>\I.
\]
Then, there exists a constant $C=C(\Om,V,\varphi,k)>0$ such that, setting
\[
t^*:=C\Biggl(\frac{1-\mu_0}{\norm{\nabla p_0}_{H^k(\Om)}+1}\Biggr)^{(k+1)k+2},
\]
for all $t'<t^*$ and $\alpha\in(0,1)$ there exist
\[
\nabla p_t\in C(0,t';C^{k-2,\alpha}(\Om,\R^2))\cap C^1(0,t';C^{k-3,\alpha}(\Om,\R^2)),
\]
and
\[
u_t=-e^{2V}\nabla^\perp \psi_t\in C(0,t';C^{k-2,\alpha}(\Om,\R^2)),
\]
solving the semigeostrophic Equation \eqref{eq:SG} in $[0,t']$. Moreover, in $[0,t']$ the constant of uniform ellipticity of $\I+\mathcal{Q}[D^2 p_t,\nabla p_t]$ is bounded away from zero, and $u_t,\nabla p_t\in L^\infty(0,t';H^k(\Om,\R^2))$.
\end{theorem}
\subsection{Structure of the paper and strategy of the proof}
In Section \ref{sec:elliptic_estimates} we start by developing the estimates of general elliptic partial differential equations with Dirichlet boundary condition in the form
\begin{equation}\label{eq:model_case_elliptic}
\begin{cases}
\diva(\calA\nabla \phi)+\calb\cdot\nabla\phi=\diva(\calF),&\text{ in }\Om\\
\phi=0,&\text{ on }\partial\Om,
\end{cases}
\end{equation}
where $\calA=\calA(x)$ is supposed uniformly elliptic, that is $\calA(x)\geq\lambda\I$ for some $\lambda>0$ and all $x\in\Om$. We will take advantage of the classic regularity theory in the Sobolev space $H^k(\Om)=\{f\in L^2(\Om):D^\alpha f\in L^2(\Om),\abs{\alpha}\leq k\}$, $k\geq 4$, to find an explicit upper bound on the constant $C>0$ realizing
\[
\norm{\nabla\phi}_{H^k(\Om)}\leq C\Bigl(\norm{\nabla\phi}_{L^2(\Om)}+\norm{\calF}_{H^k(\Om)}\Bigr),
\]
in terms of the $H^{k-1}$-norm of $\calA$, $\diva(\calA)$, $\calb$ and the elliptic constant $\lambda$. The key observation here is that if $\diva(\calA)$ shares the same regularity as $\calA$, then we gain two derivatives for the solution $\phi$ instead of one.\par Section \ref{sec:3} is devoted entirely to the construction of an approximate solution. We start by taking advantage of the conformal nature of the metric to "flatten" the Riemannian operators and see \eqref{eq:primitive_SG} as a lower order perturbation of the equation in $(\R^2,dx)$. Then, we formally obtain an elliptic partial differential equation for the potential of the velocity (recall that $\Om$ is simply connected and the fluid is incompressible) of the form \eqref{eq:model_case_elliptic} "killing" the time derivative on the rotated gradient $\partial_t\nabla^\perp p$ by applying the divergence operator on both sides of the semigeostrophic equation. In particular $\calA$ has the form $\I+\cof(\mathcal{Q})$, and here is where the stability condition comes from as a necessary requirement of solvability. A very nice cancellation property of the cofactor matrix ensures $\norm{\mathcal{Q}}_{H^{k-1}(\Om)}\sim\norm{\diva(\cof(\mathcal{Q}))}_{H^{k-1}(\Om)}$, allowing us to take full advantage of the previous general elliptic estimates. We then construct a sequence of approximate solutions regularizing the semigeostrophic equation and discretizing the time in little steps.
\par In order to prove uniform existence of a sequence of regularized solutions, in Section \ref{sec:energy} we operate an Energy Estimate on the Sobolev norm of the pressure gradient and the elliptic constant $\lambda$ of $\I+\mathcal{Q}$. Here the elliptic regularity estimate on the velocity plays a role to prove that 
\[
\abs*{\frac{d}{dt}(-\lambda)}+\abs*{\frac{d}{dt}\norm{\nabla p}_{H^k(\Om)}}\lesssim\Bigl(\frac{\norm{\nabla p}_{H^k(\Om)+1}}{\lambda}\Bigr)^{M(k)},
\]
for some exponent $M(k)>0$. A Grönwall-type argument on a well chosen function completes the proof of uniform existence local-in-time of the approximate solutions. 
\par In Section \ref{sec:compactness} we extract a smooth solution of the semigeostrophic equations by applying a suitable argument of compactness.
\subsection*{Acknowledgments}
The author would like to thank Professor A. Figalli for his guidance and constant support. The author has received funding from the European Research Council under the Grant Agreement No. 721675 “Regularity and Stability in Partial Differential Equations (RSPDE)”.\section{Explicit elliptic estimates}\label{sec:elliptic_estimates}
We refer to \cite{A18} and \cite{E15} for the classical elliptic regularity methods that we will employ. We start by stating two useful interpolation results.
\begin{lemma}\label{thm:Adams}Let $\Om\subset\R^n$ be any bounded and smooth domain. Then, for every $0\leq k\leq m$ there exists a constant $c=c(k,m,\Om)>0$ such that
\[
\norm{v}_{H^k(\Om)}\leq c\norm{v}_{L^2(\Om)}^{1-\frac{k}{m}}\norm{v}_{H^m(\Om)}^{\frac{k}{m}},
\]
for every $v\in H^m(\Om)$.
\end{lemma}
\begin{proof}The proof can be found in \cite[Chapter 5]{A08}.
\end{proof}
\begin{lemma}Let $\Om\subset \R^2$ be any smooth and bounded domain, and  let $v,w$ be functions in $H^r(\Om)$. Then, there exists $C=C(r,\Om)>0$ such that
\begin{equation}\label{eq:banach_algebra}
\norm{D^\alpha(vw)}_{L^2(\Om)}\leq C\Bigl(\norm{v}_{L^\infty(\Om)}\norm{w}_{H^r(\Om)}+\norm{w}_{L^\infty(\Om)}\norm{v}_{H^r(\Om)}\Bigr),
\end{equation}
for all multi-index $\abs{\alpha}=r$.
In particular, the following inequalities
\begin{equation}\label{eq:interpol_1}
\norm{D^\alpha (vw)-vD^\alpha w}_{L^2(\Om)}\leq C_r\Bigl(\norm{\nabla v}_{L^\infty(\Om)}\norm{w}_{H^{r-1}(\Om)}+\norm{v}_{H^r(\Om)}\norm{w}_{L^\infty(\Om)}\Bigr),
\end{equation}
and
\begin{equation}\label{eq:interpol_2}
\norm{D^\alpha(vw)-vD^\alpha w-wD^\alpha v}_{L^2(\Om)}\leq C_r\Bigl(\norm{\nabla v}_{L^\infty(\Om)}\norm{w}_{H^{r-1}(\Om)}+\norm{v}_{H^{r-1}(\Om)}\norm{\nabla w}_{L^\infty(\Om)}\Bigr),
\end{equation}
hold.
\end{lemma}
\begin{proof}The proof can be found in \cite[Lemma 3.4]{MB01}.
\end{proof}
\subsection{Set-up}
Let $\Om$ be an open, bounded subset of $\R^2$, and suppose we are given a symmetric matrix $\calA\in C^\infty(\bar\Om)^{2\times 2}$ and vector fields $\calb, \calF\in C^\infty(\bar\Om)^{2}$, such that there exists $\lambda>0$ satisfying
\[
0<\lambda\I\leq \calA.
\]
Define $\diva(\calA)\in C^\infty(\bar\Om)^2$ as
\[
\diva(\calA)^j:=\sum_{i=1}^2\partial_i\calA_{ij},
\]
such that
\[
\diva(\calA\nabla\phi)=\tr(\calA D^2\phi)+\diva(\calA)\cdot\nabla\phi,\quad\forall\phi\in C^2(\Om).
\]
Let $\phi\in C^\infty(\bar\Om)$ be solution of
\begin{equation}\label{eq:elliptic}
\begin{cases}
\diva(\calA\nabla\phi)+\calb\cdot\nabla\phi=\diva(\calF),\,&\text{in }\Om,\\
\phi=0,\,&\text{on }\partial\Om.
\end{cases}
\end{equation}
The goal of this section is to prove the following global estimate.
\begin{proposition}[Global estimates]\label{prop:global_estimates}Suppose that $\partial\Om$ is of class $C^{k+1}$ for some $k\geq 4$. Then, there exists a universal constant $C_{k,\Om}>0$ such that
\begin{equation}\label{eq:GLOBAL_EST1}
\norm{\nabla\phi}_{H^k(\Om)}\leq C_{\Om,k}\Bigl(\lambda^{-(k+1)k}\mathbf{M}^{(k+1)k}\norm{\nabla\phi}_{L^2(\Om)}+\lambda^{-(k+1)}\mathbf{M}^{k}\norm{\calF}_{H^k(\Om)}\Bigr),
\end{equation}
where 
\[
\mathbf{M}:=\Bigl(\norm{\calA}_{H^{k-1}(\Om)}+\norm{\diva(\calA)}_{H^{k-1}(\Om)}+\norm{\calb}_{H^{k-1}(\Om)}\Bigr).
\]
Moreover, if $\calb=\nabla^\perp \calf$ for some $\calf\in C^\infty(\bar\Om)$, then
\begin{equation}\label{eq:GLOBAL_EST2}
\norm{\nabla\phi}_{H^k(\Om)}\leq C_{\Om,k}\lambda^{-(k+1)k-1}\mathbf{M}^{(k+1)k}\norm{\calF}_{H^k(\Om)}.
\end{equation}
\end{proposition}
\begin{remark}\label{rmk:symmetrization}An important situation in which the particular case $\calb=\nabla^\perp\calf$ of Proposition \ref{prop:global_estimates} arises is when we symmetrize the elliptic matrix. In fact, suppose that the elliptic equation is of the form
\[
\diva\bigl((\calA+\tilde\calA)\nabla\phi\bigr)=\diva(\calF),
\]
where $\tilde\calA$ is an antisymmetric matrix. In this case we have that
\[
\diva\bigl(\tilde\calA\nabla\phi\bigr)=\diva(\tilde\calA)\cdot\nabla\phi+\tr(\tilde\calA D^2\phi)=\partial_1\tilde\calA_{12}\partial_2\phi+\partial_2\tilde\calA_{21}\partial_1\phi=\partial_1\tilde\calA_{12}\partial_2\phi-\partial_2\tilde\calA_{12}\partial_1\phi=\nabla^\perp\tilde\calA_{12}\cdot\nabla\phi,
\]
i.e. the coefficient $\calb$ comes from the rotated potential $\calf=\tilde\calA_{12}$.
\end{remark}
\subsection{Rescaled elliptic estimates}
Fix $k\geq 4$. To simplify the exposition of the following estimates, we will write
\[
a\lesssim b,\,(\text{or }a\lesssim_r b),
\]
if there exists some constant $c=c(\Om,k)>0$ (respectively $c=c(\Om,k,r)>0$),  such that
\[
\abs{a}\leq c b.
\]
In this section, we will suppose that
\begin{equation}\label{eq:renorm}
\lambda,\norm{\calA}_{H^{k-1}(\Om)},\norm{\diva(\calA)}_{H^{k-1}(\Om)},\norm{\calb}_{H^{k-1}(\Om)}\leq 1.
\end{equation}
Consequently, by Sobolev embeddings, we also have that
\[
\norm{\calA}_{W^{k-3,\infty}(\Om)},\norm{\calb}_{W^{k-3,\infty}(\Om)}\lesssim 1.
\]
We start by proving a local interior estimate.
\begin{proposition}[Rescaled interior estimates]\label{prop:interior_est}Fix $x_0\in\Om$ and $r>0$ such that $B_r:=B(x_0,r)\subset\Om$. Then, the  interior estimate
\begin{equation}\label{eq:interior_estimate}
\norm{\nabla\phi}_{H^{k}(B_{r/2})}\lesssim_r\frac{1}{\lambda}\Bigl(\norm{\nabla\phi}_{H^{k-1}(B_r)}+\norm{\calF}_{H^k(B_r)}\Bigr),
\end{equation}
holds.
\end{proposition}
\begin{proof}
Let $\abs{\alpha}=k$ be any multi-index. Then, differentiating $\alpha$-times \eqref{eq:elliptic}, we have that
\[
0=-\diva\bigl(\partial_\alpha(\calA\nabla\phi)\bigr)-\partial_\alpha(\calb\nabla\phi)+\diva(\partial_\alpha\calF),
\]
which implies, adding $\diva(\calA\partial_\alpha\nabla\phi)$ to both sides, that
\begin{equation}\label{eq:elliptic1}
\begin{split}
\diva(\calA\partial_\alpha\nabla\phi)&=\diva\bigl(\calA\partial_\alpha\nabla\phi-\partial_\alpha(\calA\nabla\phi)\bigr)-\partial_\alpha(\calb\nabla\phi)+\diva(\partial_\alpha\calF)\\
&=\diva\bigl(\partial_\alpha\calA\nabla\phi+\calA\partial_\alpha\nabla\phi-\partial_\alpha(\calA\nabla\phi)\bigr)-\partial_\alpha(\calb\nabla\phi)+\diva(\partial_\alpha\calF)-\diva(\partial_\alpha\calA\nabla\phi),
\end{split}
\end{equation}
where in the second line we simply add and subtract $\diva(\partial_\alpha\calA\nabla\phi)$. Call
\[
\calX:=\partial_\alpha\calA\nabla\phi+\calA\partial_\alpha\nabla\phi-\partial_\alpha(\calA\nabla\phi).
\]
Fix $x_0\in \Om$ and $r>0$ such that $B_r:=B(x_0,r)\subset\Om$. Choose $\eta\in C^\infty_c(B_r)$ such that $\eta\lvert_{B_{r/2}}\equiv 1$, $\eta\lvert_{\R^2\setminus B_r}\equiv0$ and $0\leq\eta\leq 1$. Testing Equation \eqref{eq:elliptic1} against $\xi:=\eta^2\partial_\alpha\phi$ gives
\[
\int\inn{\calA\partial_\alpha\nabla\phi,\nabla\xi}\,dx=\int\inn{\calX+\partial_\alpha\calF,\nabla\xi}\,dx+\int\partial_\alpha(\calb\nabla\phi)\xi\,dx+\int\diva(\partial_\alpha\calA\nabla\phi)\xi\,dx.
\]
Since $\nabla\xi=\eta^2\nabla\partial_\alpha\phi+2\partial_{\alpha}\phi\eta\nabla\eta$, taking advantage of the ellipticity of $\calA$ we can estimate
\begin{equation}\label{eq:elliptic2}
\begin{split}
\lambda\int\eta^2\abs{\partial_\alpha\nabla\phi}^2\,dx&\leq\underbrace{-\int\inn{\calA\partial_\alpha\nabla\phi,2\partial_\alpha\phi\eta\nabla\eta}\,dx}_{(I)}+\underbrace{\int\inn{\calX+\partial_\alpha\calF,\nabla\xi}\,dx}_{(II)}+\underbrace{\int\partial_\alpha(\calb\nabla\phi)\xi\,dx}_{(III)}\\
&\quad+\underbrace{\int\diva(\partial_\alpha\calA\nabla\phi)\xi\,dx}_{(IV)}.
\end{split}
\end{equation}
We will now treat (I)-(IV) separately. By the Young inequality, since $\norm{\calA}_{L^\infty(\Om)}\lesssim 1$, we have that
\[
(I)\lesssim_r\frac{1}{\epsilon}\int_{B_r}\abs{\partial_\alpha\phi}^2\,dx+\epsilon\int\eta^2\abs{\nabla\partial_\alpha\phi}^2\,dx,
\]
for every $\epsilon>0$. Now, observe that for every smooth function $h$ and $0<\epsilon\leq 1$, it holds that
\begin{equation}\label{eq:nabla_xi}
\begin{split}
\int\abs{h}\abs{\nabla\xi}\,dx&\leq\int\abs{h}\Bigl(\eta^2\abs{\nabla\partial_\alpha\phi}+2\eta\abs{\nabla\eta}\abs{\partial_\alpha\phi}\Bigr)\,dx\\
&\lesssim\frac{1}{\epsilon}\int_{B_r}\abs{h}^2\,dx+\epsilon\int\eta^2\abs{\nabla\partial_\alpha\phi}^2\,dx+\int\abs{h}\eta\abs{\nabla\eta}\abs{\partial_\alpha\phi}\,dx\\
&\lesssim_r\frac{1}{\epsilon}\int_{B_r}\abs{h}^2\,dx+\epsilon\int\eta^2\abs{\nabla\partial_\alpha\phi}^2\,dx+\int_{B_r}\abs{\partial_\alpha\phi}^2\,dx.
\end{split}
\end{equation}
Also, recalling that $k\geq 4$, by interpolation inequality \eqref{eq:interpol_2} we can easily estimate
\[
\norm{\calX}_{L^2(B_r)}\lesssim\Bigl(\norm{\calA}_{W^{1,\infty}(B_r)}\norm{\phi}_{H^{k-1}(B_r)}+\norm{\calA}_{H^{k-1}(B_r)}\norm{D^2\phi}_{L^\infty(B_r)}\Bigr)\lesssim\norm{\nabla\phi}_{H^{k-1}(B_r)}.
\]
Therefore, we obtain that
\[
(II)\lesssim_r\frac{1}{\epsilon}\norm{\nabla\phi}_{H^{k-1}(B_r)}^2+\frac{1}{\epsilon}\norm{\partial_\alpha \calF}^2_{L^2(B_r)}+\epsilon\int\eta^2\abs{\nabla\partial_\alpha\phi}^2\,dx+\int_{B_r}\abs{\partial_\alpha\phi}^2\,dx.
\]
Finally, consider the terms $(III)$ and $(IV)$. Recall that in \eqref{eq:renorm} we assumed only the $H^{k-1}(\Om)$-norms of $\calA$, $\diva(\calA)$ and $\calb$ to be controlled by 1. This means that we need to integrate by parts in such a way that these terms are differentiated at most $(k-1)$-times. Choose $i\in\{1,2\}$ such that $\partial_\alpha=\partial_\beta\partial_i$, with $\abs{\beta}=k-1$. Then
\begin{align*}
(III)&=-\int\partial_\beta(\calb\nabla\phi)\partial_i\xi\,dx=-\int\Bigl(\partial_\beta(\calb\nabla\phi)-\partial_\beta\calb\nabla\phi\Bigr)\partial_i\xi\,dx-\int(\partial_\beta\calb\nabla\phi)\partial_i\xi\,dx,
\end{align*}
which by \eqref{eq:interpol_1} and \eqref{eq:nabla_xi} gives
\[
(III)\lesssim_r\frac{1}{\epsilon}\norm{\nabla\phi}^2_{H^{k-1}(B_r)}+\epsilon\int\eta^2\abs{\nabla\partial_\alpha\phi}^2\,dx.
\]
Similarly we have that
\begin{align*}
(IV)&=\int\partial_i\diva(\partial_\beta\calA\nabla\phi)\xi-\diva(\partial_\beta\calA\partial_i\nabla\phi)\xi\,dx=-\int\diva(\partial_\beta\calA\nabla\phi)\partial_i\xi-\inn{\partial_\beta\calA\partial_i\nabla\phi,\nabla\xi}\,dx\\
&=-\int\tr(\partial_\beta\calA D^2\phi)\partial_i\xi+\diva(\partial_\beta \calA)\cdot\nabla\phi\partial_i\xi-\inn{\partial_\beta\calA\partial_i\nabla\phi,\nabla\xi}\,dx\\
&\leq\int\Bigl(\abs{\tr(\partial_\beta\calA D^2\phi)}+\abs{\partial_\beta\diva(\calA)\cdot\nabla\phi}+\abs{\partial_\beta\calA\partial_i\nabla\phi}\Bigr)\abs{\nabla\xi}\,dx\\
&\lesssim_r\frac{1}{\epsilon}\norm{\nabla\phi}^2_{W^{1,\infty}(B_r)}+\epsilon\int\eta^2\abs{\nabla\partial_\alpha\phi}^2\,dx+\int_{B_r}\abs{\partial_\alpha\phi}^2\,dx\\
&\lesssim_r \frac{1}{\epsilon}\norm{\nabla\phi}^2_{H^{k-1}(B_r)}+\epsilon\int\eta^2\abs{\nabla\partial_\alpha\phi}^2\,dx+\int_{B_r}\abs{\partial_\alpha\phi}^2\,dx.
\end{align*}
Letting $\epsilon=c_r\lambda$, for some small constant $c_r>0$, Equation \eqref{eq:elliptic2} gives
\[
\norm{\partial_\alpha\nabla\phi}_{L^2(B_{r/2})}^2\lesssim_r\frac{1}{\lambda^2}\Bigl(\norm{\nabla\phi}_{H^{k-1}(B_r)}^2+\norm{\calF}_{H^k(B_r)}^2\Bigr),
\]
as wished.
\end{proof}
Now, to obtain a similar estimate on the boundary, we start by treating the flat case.
\begin{proposition}[Rescaled flat boundary estimates]\label{prop:flat_boundary_est}Let  $\Om=B_r^+:=\{x^2>0\}\cap B(0,r)$. Then, 
\begin{equation}\label{eq:boundary_estimates}
\norm{\nabla\phi}_{H^k(B_{r/2}^+)}\lesssim_r\frac{1}{\lambda^{k+1}}\Bigl(\norm{\nabla\phi}_{H^{k-1}(B_r^+)}+\norm{\calF}_{H^k(B_r^+)}\Bigr),
\end{equation}
\end{proposition}
\begin{proof}
We first start by estimate the norm of the tangential derivatives. Fix $k\geq 4$.  Let $\eta\in C^\infty_c(\R^2)$ be a cutoff function such that $\eta\mid_{B_{r/2}}=1$, $\eta\mid_{\R^2\setminus B_r}=0$ and $0\leq\eta\leq 1$. Then, since the test function $\xi:=\eta^2\partial_1^k\phi$ vanishes on $\partial B_r^+$ (recall that we prescribed $\phi=0$ on the segment $(-r,r)\times\{0\}$), we can repeat the proof of Proposition \ref{prop:interior_est} for $\alpha=(k,0)$ obtaining the estimate
\[
\norm{\nabla\partial_\alpha\phi}_{L^2(B_{r/2}^+)}^2=\norm{\nabla\partial_1^k\phi}_{L^2(B_{r/2}^+)}^2\lesssim_r \frac{1}{\lambda^2}\Bigl(\norm{\nabla\phi}_{H^{k-1}(B_r^+)}^2+\norm{\calF}_{H^k(B_r^+)}^2\Bigr).
\]
We now show that for all multi-index $\alpha=\alpha_l:=(k-l,l)$ and $l=0,\dots k$, we can estimate
\[
\norm{\nabla\partial_{\alpha_l}\phi}_{L^2(B_{r/2}^+)}\lesssim_r \frac{1}{\lambda^{2(l+1)}}\Bigl(\norm{\nabla\phi}_{H^{k-1}(B_r^+)}^2+\norm{\calF}_{H^k(B_r^+)}^2\Bigr).
\]
 We proceed by induction over $l$: we have already treated the case $l=0$. Then, suppose the claim true for all $0\leq l'\leq l$, for some fixed $1\leq l<k$. We have to check the case with $\alpha_{l+1}=(k-(l+1),l+1)$. Let $\gamma_1=(k-l,l-1)$ and $\gamma_2=(k-(l+1),l)$, so that
\[
\nabla\partial_{\alpha_{l+1}}=\partial_{22}(\partial_{\gamma_1},\partial_{\gamma_2}).
\]
We will now take advantage of Equation \eqref{eq:elliptic}: after differentiation and suitable rearrangement, for $s=1,2$ we have that
\[
\tr(\calA\partial_{\gamma_s} D^2\phi)=\Bigl(\tr(\calA\partial_{\gamma_s} D^2\phi)-\partial_{\gamma_s}\tr(\calA D^2\phi)\Bigr)-\partial_{\gamma_s}\Bigl(\bigl(\calb+\diva(\calA)\bigr)\cdot\nabla\phi\Bigr)-\diva(\partial_{\gamma_s} \calF),
\]
which, developing the trace, becomes
\begin{align*}
\calA_{22}\partial_{22}\partial_{\gamma_s}\phi&=\Bigl(\tr(\calA\partial_{\gamma_s} D^2\phi)-\partial_{\gamma_s}\tr(\calA D^2\phi)\Bigr)-\partial_{\gamma_s}\Bigl(\bigl(\calb+\diva(\calA)\bigr)\cdot\nabla\phi\Bigr)-\diva(\partial_{\gamma_s}\calF)\\
&\quad-\sum_{(i,j)\neq(2,2)}\calA_{ij}\partial_{ij}\partial_{\gamma_s}\phi.
\end{align*}
Since $\calA$ is elliptic, then the coefficient $\calA_{22}$ is controlled uniformly from below by the elliptic constant $\lambda$. Therefore, applying the $L^2$ norm over $B_{r/2}^+$ on both sides, and taking advantage of interpolation inequalities \eqref{eq:banach_algebra} and \eqref{eq:interpol_1} we obtain the estimate
\begin{align*}
\lambda\norm{\nabla\partial_{\alpha_{l+1}}\phi}_{L^2(B_{r/2}^+)}&\lesssim_r \norm{\calA}_{W^{1,\infty}(B_{r/2}^+)}\norm{D^2\phi}_{H^{k-2}(B_{r/2}^+)}+\norm{\calA}_{H^{k-1}(B_{r/2}^+)}\norm{D^2\phi}_{L^\infty(B_{r/2}^+)}\\
&\quad+\norm{\calb+\diva(\calA)}_{L^{\infty}(B_{r/2}^+)}\norm{\nabla\phi}_{H^{k-1}(B_{r/2}^+)}+\norm{\calb+\diva(\calA)}_{H^{k-1}(B_{r/2}^+)}\norm{\nabla\phi}_{L^\infty(B_{r/2}^+)}\\
&\quad+\norm{\calF}_{H^k(B_{r/2}^+)}+\norm{\calA}_{L^\infty(B_{r/2}^+)}\sum_{(i,j)\neq(2,2)}\Bigl(\norm{\partial_{ij}\partial_{\gamma_1}\phi}_{L^2(B_{r/2}^+)}+\norm{\partial_{ij}\partial_{\gamma_2}\phi}_{L^2(B_{r/2}^+)}\Bigr)\\
&\lesssim_r\norm{\nabla\phi}_{H^{k-1}(B_{r/2}^+)}+\norm{\calF}_{H^k(B_{r/2}^+)}+\sum_{(i,j)\neq(2,2)}\Bigl(\norm{\partial_{ij}\partial_{\gamma_1}\phi}_{L^2(B_{r/2}^+)}+\norm{\partial_{ij}\partial_{\gamma_2}\phi}_{L^2(B_{r/2}^+)}\Bigr).
\end{align*}
Now, since
\[
\nabla\partial_{\alpha_{l-1}}=(\partial_{11}\partial_{\gamma_1},\partial_{12}\partial_{\gamma_1}),
\]
and
\[
\nabla\partial_{\alpha_l}=(\partial_{11}\partial_{\gamma_2},\partial_{12}\partial_{\gamma_2}),
\]
we can control
\[
\sum_{(i,j)\neq(2,2)}\Bigl(\norm{\partial_{ij}\partial_{\gamma_1}\phi}_{L^2(B_{r/2}^+)}+\norm{\partial_{ij}\partial_{\gamma_2}\phi}_{L^2(B_{r/2}^+)}\Bigr)\leq 3\Bigl(\norm{\nabla\partial_{\alpha_l}\phi}_{L^2(B_{r/2}^+)}+\norm{\nabla\partial_{\alpha_{l-1}}\phi}_{L^2(B_{r/2}^+)}\Bigr),
\]
which by induction gives
\[
\lambda\norm{\nabla\partial_{\alpha_{l+1}}\phi}_{L^2(B_{r/2}^+)}\lesssim_r \Bigl
(1+\frac{1}{\lambda^{l}}+\frac{1}{\lambda^{l+1}}\Bigr)\bigl(\norm{\nabla\phi}_{H^{k-1}(B_r)}+\norm{\calF}_{H^k(B_r)}\bigr),
\]
and hence
\[
\norm{\nabla\partial_{\alpha_{l+1}}\phi}_{L^2(B_{r/2}^+)}\lesssim_r \frac{1}{\lambda^{l+2}}\Bigl(\norm{\nabla\phi}_{H^{k-1}(B_r)}+\norm{\calF}_{H^k(B_r)}\Bigr),
\]
as wished, completing the induction.
\end{proof}
Now we prove that we can recover the same estimate for a domain with curved boundary.
\begin{proposition}[Rescaled curved boundary estimates]\label{prop:curved_estimate}Let $\Om\subset\R^2$ be any open domain with boundary of class $C^{k+1}$. Choose $x_0\in\partial\Om$ and $r>0$ sufficiently small such that there exists a $C^{k+1}$-diffeomorphism 
\[
\Phi: B(x_0,r)\cap\Om\to B_r^+,
\]
with inverse $\Psi=\Phi^{-1}$, such that $\Phi(\partial\Om\cap B(x_0,r))=(-r,r)\times\{0\}$, and $\det(D\Phi)=1$. Then, calling $U_r^+:=B(x_0,r)\cap\Om$, we have that there exists $C_{\Phi}>0$ such that
\begin{equation}\label{eq:curved_boundary_estimates}
\norm{\nabla\phi}_{H^k(U_{r/2}^+)}\lesssim_r\frac{C_\Phi}{\lambda^{k+1}}\Bigl(\norm{\nabla\phi}_{H^{k-1}(U_r^+)}+\norm{\calF}_{H^k(U_r^+)}\Bigr).
\end{equation}
\end{proposition}
\begin{proof}
One can check directly that
\begin{align*}
\phi'(y)&:=\phi(\Psi(y)),\\
\calA'(y)_{rs}&:=\sum_{ij}\calA(\Psi(y))_{ij}\partial_{x_i}\Phi^r(\Psi(y))\partial_{x_j}\Phi^s(\Psi(y)),\\
\calb'(y)^{r}&:=\sum_{i}\calb(\Psi(y))^i\partial_{x_i}\Phi^r(\Psi(y)),\\
\calF'(y)^{r}&:=\sum_{i}\calF^i(\Psi(y))\partial_{x_i}\Phi^r(\Psi(y)),
\end{align*}
solves
\[
\begin{cases}
\diva(\calA'\nabla\phi')+\calb'\cdot\nabla\phi'=\diva(\calF'),\,&\text{in }B_r^+,\\
\phi'=0,\,&\text{on }(-r,r)\times\{0\}.
\end{cases}
\]
and $\lambda\I\leq\calA'$. We have to compute $\diva(\calA')(y)^s=\sum_{r}\partial_{y_r}\calA'(y)_{rs}$ in terms of $\calA$ and $\diva(\calA)$. Now,
\begin{align*}
\diva(\calA')(y)^s&=\sum_{r}\partial_{y_r}\calA'(y)_{rs}=\sum_{r}\partial_{y_r}\Bigl(\sum_{ij}\calA(\Psi(y))_{ij}\partial_{x_i}\Phi^r(\Psi(y))\partial_{x_j}\Phi^s(\Psi(y))\Bigr)\\
&=\sum_{ijrs}\partial_{x_s}\calA(\Psi(y))_{ij}\partial_{y_r}\Psi^s(y)\partial_{x_i}\Phi^r(\Psi(y))\partial_{x_j}\Phi^s(\Psi(y))\\
&\quad+\sum_{ijr}\calA(\Psi(y))_{ij}\partial_{y_r}\Bigl(\partial_{x_i}\Phi^r(\Psi(y))\partial_{x_j}\Phi^s(\Psi(y))\Bigr),
\end{align*}
and since $\sum_r\partial_{y_r}\Psi^s(y)\partial_{x_i}\Phi^r(\Psi(y))=\delta_{si}$, it follows that
\begin{align*}
\diva(\calA')(y)^s&=\sum_{ijs}\partial_{x_i}\calA(\Psi(y))_{ij}\partial_{x_j}\Phi^s(\Psi(y))+\sum_{ijr}\calA(\Psi(y))_{ij}\partial_{y_r}\Bigl(\partial_{x_i}\Phi^r(\Psi(y))\partial_{x_j}\Phi^s(\Psi(y))\Bigr)\\
&=\sum_{js}\diva(\calA)(\Psi(y))^{j}\partial_{x_j}\Phi^s(\Psi(y))+\sum_{ijr}
\calA(\Psi(y))_{ij}\partial_{y_r}\Bigl(\partial_{x_i}\Phi^r(\Psi(y))\partial_{x_j}\Phi^s(\Psi(y))\Bigr).
\end{align*}
Therefore, there exists $C_\Phi>0$ such that
\[
\norm{\diva(\calA')}_{H^{k-1}(B_r^+)}\leq C_\Phi\Bigl(\norm{\calA}_{H^{k-1}(U^+_r)}+\norm{\diva(\calA)}_{H^{k-1}(U^+_r)}\Bigr).
\]
It suffices to apply Proposition \ref{prop:flat_boundary_est} in order to complete the proof.
\end{proof}
By covering $\Om$ with sufficiently small balls, we can prove a global estimate for the rescaled elliptic equation.
\begin{proposition}[Rescaled global estimates]\label{prop:rescaled_global_estimates}
There exists $C_{\Om,k}>0$ such that
\begin{equation}\label{eq:rescaled_global_est}
\norm{\nabla\phi}_{H^k(\Om)}\leq C_{\Om,k}\Bigl(\lambda^{-(k+1)k}\norm{\nabla\phi}_{L^2(\Om)}+\lambda^{-(k+1)}\norm{\calF}_{H^k(\Om)}\Bigr).
\end{equation}
Moreover, if there exists $\calf\in C^\infty(\bar\Om)$ such that $\calb=\nabla^\perp\calf$, then
\begin{equation}\label{eq:rescaled_global_est2}
\norm{\nabla\phi}_{H^k(\Om)}\leq C_{\Om,k}\lambda^{-(k+1)k-1}\norm{\calF}_{H^k(\Om)}.
\end{equation}
\end{proposition}
\begin{proof}
Covering $\Om$ by sufficiently many balls, combining Propositions \ref{prop:interior_est} and \ref{prop:curved_estimate}, it follows that for any $k\geq 4$ there exists $C_{k,\Om}>0$ such that
\begin{equation*}
\norm{\nabla\phi}_{H^k(\Om)}\leq\frac{C_{k,\Om}}{\lambda^{k+1}}\Bigl(\norm{\nabla\phi}_{H^{k-1}(\Om)}+\norm{\calF}_{H^k(\Om)}\Bigr).
\end{equation*}
We distinguish two cases: if $\norm{\nabla\phi}_{H^{k-1}(\Om)}\leq\norm{\calF}_{H^k(\Om)}$, then
\begin{equation*}
\norm{\nabla\phi}_{H^k(\Om)}\leq\frac{2C_{k,\Om}}{\lambda^{k+1}}\norm{\calF}_{H^k(\Om)},
\end{equation*}
and we are done. Otherwise, since
\begin{equation*}
\norm{\nabla\phi}_{H^k(\Om)}\leq\frac{2C_{k,\Om}}{\lambda^{k+1}}\norm{\nabla\phi}_{H^{k-1}(\Om)},
\end{equation*}
the interpolation inequality of Theorem \ref{thm:Adams} implies that there exist $C'_{k,\Om}>0$ such that
\[
\norm{\nabla\phi}_{H^{k-1}(\Om)}\leq C'_{k,\Om}\norm{\nabla\phi}_{L^2(\Om)}^{1-\frac{k-1}{k}}\Bigl(\frac{1}{\lambda^{k+1}}\norm{\nabla\phi}_{H^{k-1}(\Om)}\Bigr)^{\frac{k-1}{k}},
\]
and hence
\[
\norm{\nabla\phi}_{H^{k-1}(\Om)}\leq ({C'}_{k,\Om})^{k}\norm{\nabla\phi}_{L^2(\Om)}\lambda^{-(k+1)(k-1)}.
\]
Finally, in both cases we have proven that there exists $C''_{k,\Om}>0$ such that
\[
\norm{\nabla\phi}_{H^k(\Om)}\leq C''_{k,\Om}\Bigl(\lambda^{-(k+1)k}\norm{\nabla\phi}_{L^2(\Om)}+\lambda^{-(k+1)}\norm{\calF}_{H^k(\Om)}\Bigr).
\]
If $\calb=\nabla^\perp\calf$, then one can get rid of the $L^2$-norm of $\nabla\phi$ simply testing \eqref{eq:elliptic} against $\phi$ and computing
\begin{align*}
\lambda\norm{\nabla\phi}^2_{L^2(\Om)}&\leq \int(\nabla^\perp\calf\cdot\nabla\phi)\phi\,dx+\int\calF\cdot\nabla\phi\,dx\\
&=\frac{1}{2}\int\nabla^\perp\calf\cdot\nabla(\phi^2)\,dx+\norm{\calF}_{L^2(\Om)}\norm{\nabla\phi}_{L^2(\Om)}\\
&=\frac{1}{2}\int_{\partial\Om}\phi^2(\nabla^\perp\calf\cdot\nu)\,dx-\frac{1}{2}\int\diva(\nabla^\perp\calf)\phi^2\,dx+\norm{\calF}_{L^2(\Om)}\norm{\nabla\phi}_{L^2(\Om)}\\
&=\norm{\calF}_{L^2(\Om)}\norm{\nabla\phi}_{L^2(\Om)}.
\end{align*}
Hence, plugging $\lambda\norm{\nabla\phi}_{L^2(\Om)}\leq\norm{\calF}_{L^2(\Om)}$ in \eqref{eq:rescaled_global_est} we finally obtain that
\[
\norm{\nabla\phi}_{H^k(\Om)}\leq C_{k,\Om}\Bigl(\lambda^{-(k+1)k-1}\norm{\calF}_{L^2(\Om)}+\lambda^{-(k+1)}\norm{\calF}_{H^k(\Om)}\Bigr),
\]
finishing the proof of the proposition (recall that  by  hypothesis $\lambda\leq 1$).
\end{proof}
We can now easily prove the main result of this section.
\begin{proof}[Proof of Proposition \ref{prop:global_estimates}]Renormalizing Equation \eqref{eq:elliptic} by dividing both sides by $\mathbf{M}$ we obtain, applying Proposition \eqref{prop:rescaled_global_estimates}, that
\[
\norm{\nabla\phi}_{H^k(\Om)}\leq C_{k,\Om}\Bigl(\Bigl(\frac{\lambda}{\mathbf{M}}\Bigr)^{-(k+1)k}\norm{\nabla\phi}_{L^2(\Om)}+\Bigl(\frac{\lambda}{\mathbf{M}}\Bigr)^{-(k+1)}\norm*{\frac{\calF}{\mathbf{M}}}_{H^k(\Om)}\Bigr),
\]
which gives \eqref{eq:GLOBAL_EST1}. The same shows \eqref{eq:GLOBAL_EST2}.
\end{proof}

\section{Local-in-time existence of smooth solutions in Eulerian coordinates}\label{sec:3}
\subsection{Flattening}

We would like to look at \eqref{eq:SG} as a perturbation of the semigeostrophic equation on the flat plane.
Since $g$ is conformal, we know that
\[
\nabla_g h=e^{2V}\nabla h,\quad\text{for all } h\in C^1(\Om),
\]
and
\begin{equation}\label{eq:flat_covariant}
D^g_XY=(X\cdot\nabla)Y-dV(X)Y-dV(Y)X+\inn{X,Y}\nabla V=D^gY\cdot X,\quad \text{ for all }X,Y\in C^1(\Om,\R^2),
\end{equation}
where $D^gY:=DY-Y\otimes\nabla V+\nabla V\otimes Y-\inn{Y,\nabla V}\I$. Since by hypothesis $u$ is divergence free, tangent to $\partial\Om$ and $\Om$ is simply connected, we can suppose that there exists some  potential $\psi$ such that
\[
u=-\nabla_g^\perp\psi=-e^{2V}\nabla^\perp \psi=:e^{2V}v,\text{ and }\psi\mid_{\partial\Om}=0.
\]
Converting all curved gradients into flat ones, substituting $u$ with $e^{2V}v$ and multiplying Equation \eqref{eq:SG} by $e^{-\varphi-2V}$ we obtain that
\begin{equation}\label{eq:SG2}
\partial_t\nabla^\perp p+e^{-\varphi}D^g_v(e^{\varphi+2V}\nabla^\perp p)+e^{-2\varphi}v^\perp+e^{-\varphi}\nabla p=0.
\end{equation}
Thanks to Equation \eqref{eq:flat_covariant} we can write
\begin{align*}
e^{-\varphi}D^g_v(e^{\varphi+2V}\nabla^\perp p)&=e^{-\varphi}J\bigl(e^{\varphi+2V}D_v^g\nabla p+e^{\varphi+2V}\inn{\nabla\varphi+2\nabla V,v}\nabla p\bigr)\\
&=e^{2V}J\bigl(D_g\nabla p+\nabla p\otimes(\nabla\varphi+2\nabla V)\bigr)v\\
&=e^{2V}\cof(D_g\nabla p+\nabla p\otimes(\nabla\varphi+2\nabla V))\nabla\psi\\
&=e^{2V}\cof(D^2p+\calB[\nabla p]),
\end{align*}
where with $\cof(\cdot)$ we denote the cofactor matrix, which in two dimensions is simply given by 
\[
\cof(M):=-JMJ,
\]
and for every vector field $\xi$ we set
\begin{align*}
\calB[\xi]&:=\nabla V\otimes\xi+\xi\otimes\nabla V-\inn{\xi,\nabla V}\I+\xi\otimes\nabla \varphi.
\end{align*}
Plugging this in Equation \eqref{eq:SG2} we finally obtain the semigeostrophic equation with flattened operators
\begin{equation}\label{eq:SGflat}
\begin{cases}
\partial_t\nabla^\perp p+e^{2V}\cof(D^2 p+\calB[\nabla p])\nabla \psi+e^{-2\varphi}\nabla\psi=-e^{-\varphi}\nabla p,&\text{ in }\Om,\\
\psi=0,&\text{ on }\partial\Om.
\end{cases}
\end{equation}

\subsection{An elliptic PDE for the velocity vector field}\label{sec_elliptic_pde}

Applying the divergence operator on both sides of  Equation \eqref{eq:SGflat}, we remove the explicit dependencies on the time variable, obtaining
\[
\diva\Bigl(e^{2V}\cof(D^2 p+\calB[\nabla p]+e^{-2\varphi-2V}\I)\nabla\psi\Bigr)=-\diva(e^{-\varphi}\nabla p).
\]
In order to rewrite this as a classical elliptic equation in divergence form,  we decompose $\calB$ into its symmetric and antisymmetric part as
\[
\calB[\xi]=\underbrace{(\nabla V+\nabla\varphi/2)\otimes\xi+\xi\otimes(\nabla V+\nabla\varphi/2)-\inn{\xi,\nabla V}\I}_{=:\calB^s[\xi]}+\underbrace{\frac{1}{2}(\xi\otimes\nabla\varphi-\nabla\varphi\otimes\xi)}_{=:\calB^{as}[\xi]}.
\]
Hence, we obtain the equation
\[
\diva\Bigl(e^{2V}\cof(D^2 p+\calB^s[\nabla p]+e^{-2\varphi-2V}\I)\nabla\psi\Bigr)+\nabla^\perp(e^{2V}\calB^{as}_{12}[\nabla p])\cdot\nabla \psi=-\diva(e^{-2\varphi}\nabla p),
\]
(see Remark \ref{rmk:symmetrization}). Finally, to simplify the exposition, define
\begin{align*}
\calQ[D\xi,\xi]&:=e^{2V+2\varphi}\cof(D\xi^T+\calB^s[\xi]),\\
\calf[\xi]&:=e^{2V} \calB^{as}_{12}[\xi],\\
\calF[\xi]&:=-e^{-\varphi}\xi,
\end{align*}
so that we can rewrite the equation as
\begin{equation}\label{eq:elliptic_u}
\begin{cases}
\diva(e^{-2\varphi}(\I+\calQ[D^2p,\nabla p])\nabla\psi)+\nabla^\perp(\calf[\nabla p])\cdot\nabla \psi=\diva(\calF[\nabla p]),&\text{ in }\Om\\
\psi=0,&\text{ on }\partial\Om.
\end{cases}
\end{equation}
Notice that in the definition of $\calQ$ we decided to transpose the matrix $D\xi$. This has clearly no effect when $D\xi=D^2p$, but it will be important to obtain the suitable cancellation of terms in the following useful lemma.
\begin{lemma}[Basic estimates on the coefficients]\label{lem:basic_est}For every $k\geq 0$ there exists a constant $C=C(\varphi,V,k)>0$ such that for every smooth vector field $\xi$ on $\Om$ the following estimates hold:
\begin{align*}
\norm{\calB[\xi]}_{H^k(\Om)}+\norm{\calF[\xi]}_{H^k(\Om)}+\norm{\calf[\xi]}_{H^k(\Om)}&\leq C\norm{\xi}_{H^k(\Om)},\\
\norm{\calQ[D\xi,\xi]}_{H^k(\Om)}+\norm{\diva(\calQ[D\xi,\xi])}_{H^k(\Om)}&\leq C\norm{\xi}_{H^{k+1}(\Om)}.\\
\end{align*}
\end{lemma}
\begin{proof}
The first four inequalities follow immediately from the definition of $\calB[\xi]$. To check the last one, simply observe that the only problematic term in $\calQ[D\xi,\xi]$ is $\cof(D\xi^T)$. Conclude by noticing that the cofactor matrix of the transpose jacobian matrix enjoys the following nice property
\[
\diva\Bigl(\cof(D\xi^T)\Bigr)=\sum_{i,j}\partial_i(\cof(D\xi^T))_{ij}=\partial^2_{12}\xi^2-\partial^2_{12}\xi^1-\partial^2_{21}\xi^2+\partial^2_{21}\xi^1=0.
\]
\end{proof}

\subsection{Discrete construction and local-in-time uniformly existence of regularized solutions}

Before presenting the algorithm to construct an approximate solution, we need to fix some notation. For all  vector field $X\in H^k(\Om,\R^2)$, consider the unique Helmholtz-Hodge orthogonal decomposition
\[
X=w+\nabla q,
\]
where $\diva(w)=0$. From now on, we denote with 
\[
\hel(X):=\nabla q,
\]
the orthogonal complement of the classical Leray projector. Explicitly, $q$ solves the Neumann-type elliptic problem $\Delta q=\diva(X)$ in $\Om$, $\partial_\nu q=X\cdot\nu$ on $\partial\Om$. With $\mol$ we denote the standard mollification
\[
\mol h:=\eta_\epsilon\ast h,\,\forall h\in L^2(\Om),L^2(\Om,\R^2),L^2(\Om,\R^{2\times 2}),\dots
\]
where $\eta_\epsilon$ is any smooth convolution kernel. We address the reader to \cite[Appendix C]{E15} and \cite[Chapter 4]{MB01} for a brief recall of the principal properties and definitions of $\mol$ and $\hel$. Fix now $k\geq 4$ and suppose we are given $\nabla p_0\in H^k(\Om,\R^2)$ such that
\[
\I+\calQ[D^2 p_0,\nabla p_0]\geq (1-\mu_0)\I>0,
\]
for some $\mu_0<1$. Choose a coefficient of mollification $\epsilon>0$ and a time step $\tau>0$. We set $\nabla p^{0}_0:=\nabla p_0$ and solve for $i=-1,0,1,\dots$ and $s\in[0,\tau]$ the system
\begin{equation}\label{eq:discrete_p}
\begin{cases}
\partial_s\nabla  p_s^{i+1}=\mathcal{F}_{\psi^{i+1}}^\epsilon(\nabla p^{i+1}_s):=\hel\mol\Bigl(e^{2V}(\mol D^2 p^{i+1}_s+\calB[\mol\nabla p^{i+1}_s]+e^{-2\varphi-2V}\I)\nabla^\perp\psi^{i+1}+e^{-\varphi}\mol\nabla^\perp p^{i+1}_s\Bigr),\\
\nabla p_0^{i+1}=\nabla p^i_\tau.
\end{cases}
\end{equation}
where $\psi^{i+1}$ is given by
\begin{equation}\label{eq:discrete_u}
\begin{cases}
\diva\Bigl(e^{-2\varphi}(\I+\calQ[\mol  D^2p^{i+1}_0,\mol \nabla p^{i+1}_0])\nabla\psi^{i+1}\Bigr)+\nabla^\perp(\calf[\mol \nabla p^{i+1}_0])\nabla\psi^{i+1}=\diva\bigl(\calF[\mol\nabla p^{i+1}_0]\bigr),&\text{ in }\Om\\
\psi^{i+1}=0,&\text{ on }\partial\Om.
\end{cases}
\end{equation}
Notice that \eqref{eq:discrete_p} and \eqref{eq:discrete_u} are nothing else than a regularized version of Equations \eqref{eq:SGflat} and \eqref{eq:elliptic_u}, where $\nabla p^i_s$ evolves continuously on each time-step solving an ordinary differential equation of the form $\dot y=F(y)$ (we take the velocity constant on each interval $[i\tau,(i+1)\tau)$), and $\psi^{i+1}$ evolves discretely as a solution of an elliptic equation. Our next goal is to prove that there exists a fixed interval of existence $[0,t^*)$ so that for every $\epsilon>0$ and $\tau=t^*/N$, for $N\in\N$ big enough, the sequence $\{\nabla p_s^{i},\nabla\psi^{i}\}_{i=0}^{N-1}$ exists.
Solvability of System \eqref{eq:discrete_p} is ensured by the following proposition.
\begin{proposition}Let $k\geq 2$, and $\epsilon>0$. Then, for every $\nabla q_0$ in $H^k(\Om,\R^2)$ and $\nabla\phi\in L^\infty(\Om,\R^2)$, there exists a global solution $\nabla q^\epsilon_s\in C^1(\R,H^k(\Om,\R^2))$ of the following partial differential equation
\[
\begin{cases}
\partial_s\nabla q_s^\epsilon=\mathcal{F}_{\phi}^\epsilon(\nabla q_s)=\hel\mol\Bigl(e^{2V}(\mol D^2 q^\epsilon_s+\calB[\mol\nabla q^\epsilon_s]+e^{-2\varphi-2V}\I)\nabla^\perp\phi+e^{-\varphi}\mol\nabla^\perp q^{\epsilon}_s\Bigr),\\
\nabla q^\epsilon_0=\nabla q_0.
\end{cases}
\]
\end{proposition}
\begin{proof}This is a direct application of the Cauchy-Lipschitz Theorem in the Banach space
\[
\mathcal{X}:=\Bigl\{\nabla  q: q\in H^{k+1}(\Om)\Bigr\}\subset H^{k}(\Om,\R^2).
\]
In fact, thanks to the Helmholtz-Hodge decomposition, it is clear that $\mathcal{F}_{\phi}^\epsilon$ maps $\mathcal{X}$ into itself. We just need to check that it is Lipschitz continuous. Let $\nabla q$ and $\nabla h$ elements in $\mathcal{X}$. Then, thanks to the properties of $\mol$ and $\hel$, we can estimate
\begin{align*}
\norm{\mathcal{F}_{\phi}^\epsilon(\nabla q)&-\mathcal{F}_{\phi}^\epsilon(\nabla h)}_{H^k(\Om)}\\
&\leq \frac{C}{\epsilon^{k}}\norm{e^{2V}\bigl(\mol(D^2 q-D^2 h)+\calB[\mol(\nabla q-\nabla h)]\bigr)\nabla^\perp\phi+e^{-\varphi}\mol(\nabla q-\nabla h)^\perp}_{L^2(\Om)}\\
&\leq \frac{C}{\epsilon^k}\norm{e^{2V}}_{\infty}\norm{\nabla\phi}_{L^\infty(\Om)}\Bigl(\norm{D^2q-D^2h}_{L^2(\Om)}+\norm{\calB[\mol(\nabla q-\nabla h)]}_{L^2(\Om)}\Bigr)\\
&\quad+\frac{C}{\epsilon^k}\norm{e^{-\varphi}}_\infty\norm{\nabla q-\nabla h}_{L^2(\Om)}.
\end{align*}
Now, thanks to Lemma \ref{lem:basic_est} we know that $\calB[\cdot]$ is a continuous functional in $L^2(\Om,\R^2)$ implying that there exists $C'=C'(V,\varphi,\Om)>0$ such that
\[
\Lip(\mathcal{F}_{\phi}^\epsilon)\leq\frac{C'}{\epsilon^k}\Bigl(\norm{\nabla\phi}_{L^\infty(\Om)}+1\Bigr)<+\infty,
\]
as wished.
\end{proof}
System \eqref{eq:discrete_u} is solvable at the step $(i+1)$ if the eigenvalue
\[
-\mu^{i+1}_s:=\inf_{\abs{\xi}=1,x\in\Om}\Bigl\{\inn{\calQ[\mol  D^2p^{i+1}_s,\mol \nabla p^{i+1}_s](x)\xi,\xi}\Bigr\},
\]
is strictly greater than $-1$ at time $s=0$. To analyse the behaviour of $\mu^{i+1}$, define
\[
-\mu^{i+1}_s(x):=\inf_{\abs{\xi}=1}\Bigl\{\inn{\calQ[\mol  D^2p^{i+1}_s,\mol \nabla p^{i+1}_s](x)\xi,\xi}\Bigr\},\quad x\in\Om,\,s\in[0,\tau].
\]
Since $\nabla p^{i+1}_s\in C^1([0,\tau],H^k(\Om,\R^2))$ we have that fixing $x$, $s\mapsto\mu^{i+1}_s(x)$ is a locally Lipschitz map, and therefore $\mu^{i+1}_s$, being the infimum  over $x\in\Om$, is  also locally Lipschitz  and hence almost everywhere differentiable.
\begin{lemma}[Dynamics of the elliptic constant]There exists $C=C(V,\varphi,\Om)>0$ such that
\begin{equation}\label{eq:LAMBDA_EST}
\frac{d}{ds}\Big\rvert_{s=s_0}(1-\mu^{i+1}_s)\geq -C\Bigl(\norm{\nabla p^{i+1}_{s_0}}_{H^4(\Om)}+1\Bigr)\norm{\nabla \psi^{i+1}}_{H^3(\Om)}-C\norm{\nabla p^{i+1}_{s_0}}_{H^3(\Om)},
\end{equation}
almost every  $s_0$ in $(0,\tau)$.
\end{lemma}
\begin{proof}Take $\delta\neq 0$ small, $x\in\Om$ and $s_0\in(0,\tau)$. Then, let $\xi_\delta\in \R^2$ be the unit vector realizing 
\[
-\mu^{i+1}_{s_0+\delta}(x)=\inn{\calQ[\mol D^2 p_{s_0+\delta}^{i+1}(x),\mol \nabla p_{s_0+\delta}^{i+1}(x)]\xi_{\delta},\xi_{\delta}}.
\]
Then,
\begin{align*}
\mu^{i+1}_{s_0}(x)-\mu^{i+1}_{s_0+\delta}(x)&\geq\inn{\Bigl(\calQ[\mol D^2 p_{s_0+\delta}^{i+1}(x),\mol \nabla p_{s_0+\delta}^{i+1}(x)]-\calQ[\mol D^2 p_{s_0}^{i+1}(x),\mol \nabla p_{s_0}^{i+1}(x)]\Bigr)\xi_\delta,\xi_{\delta}}\\
&=\inn{\int_{s_0}^{s_0+\delta}\partial_t\calQ[\mol D^2 p_{t}^{i+1}(x),\mol \nabla p_{t}^{i+1}(x)]\,dt\cdot\xi_\delta,\xi_\delta}\\
&\geq-\int_{s_0}^{s_0+\delta}\norm{\calQ[\mol D^2\partial_t p_{t}^{i+1}(x),\mol \nabla \partial_t p_{t}^{i+1}(x)]}_{L^\infty(\Om)}.
\end{align*}
By the Sobolev embedding of $L^\infty(\Om)$ in $H^2(\Om)$ and by Lemma \ref{lem:basic_est}, we obtain that there exists $C_1>0$ such that
\[
\mu^{i+1}_{s_0}(x)-\mu^{i+1}_{s_0+\delta}(x)\geq -C_1\int_{s_0}^{s_0+\delta}\norm{\partial_t\nabla p_t^{i+1}}_{H^3(\Om)}\,dt.
\]
Finally, thanks to the discrete construction of the pressure gradient given by Equation \eqref{eq:discrete_p}, the fact that $H^3(\Om)$ is a Banach Algebra, we conclude that there exists $C>0$ such that
\[
\mu^{i+1}_{s_0}(x)-\mu^{i+1}_{s_0+\delta}(x)\geq -C\int_{s_0}^{s_0+\delta}\Bigl(\norm{\nabla p^{i+1}_{t}}_{H^4(\Om)}+1\Bigr)\norm{\nabla \psi^{i+1}}_{H^3(\Om)}-C\norm{\nabla p^{i+1}_{t}}_{H^3(\Om)}\,dt.
\]
The result follows by dividing everything by $\delta$, and letting $\delta$ go to zero.
\end{proof}

\section{Energy estimates}\label{sec:energy}
\begin{proposition}[Energy estimates]Let $k\geq 4$. Then, there exists $C=C(k,\Om,V,\varphi)>0$ such that
\begin{equation}\label{eq:P_EST}
\frac{d}{ds}\norm{\nabla p^{i+1}_s}_{H^k(\Om)}\leq C\Bigl(\norm{\nabla p^{i+1}_s}_{H^k(\Om)}+1\Bigr)\norm{\nabla \psi^{i+1}}_{H^k(\Om)}+C\norm{\nabla p^{i+1}_s}_{H^k(\Om)}.
\end{equation}
\end{proposition}
\begin{proof}
Fix any multi index $\abs{\alpha}\leq k$. Since the operators $\mol$ and $\hel$ commute and are self-adjoint with respect to the $L^2$-product, we  can compute
\begin{align*}
\frac{d}{ds}\frac{1}{2}\int&\abs{\partial_\alpha\nabla p^{i+1}_s}^2\,dx=\int\inn{\partial_\alpha\nabla p^{i+1}_s,\partial_\alpha\partial_s\nabla p^{i+1}_s}\,dx\\
&=\int\inn{\mol\partial_\alpha\nabla p^{i+1}_s,\partial_\alpha\Bigl(e^{2V}(\mol D^2 p^{i+1}_s+\calB[\mol\nabla p^{i+1}_s]+e^{-2\varphi-2V}\I)\nabla^\perp\psi^{i+1}+e^{-\varphi}\mol\nabla^\perp p^{i+1}_s\Bigr)}\,dx.
\end{align*}
Set $P_s:=\mol\nabla p^{i+1}_s$. There exists $C_\varphi>0$ such  that
\begin{align*}
\frac{d}{ds}&\frac{1}{2}\int\abs{\partial_\alpha\nabla p^{i+1}_s}^2\,dx=\int\inn{\partial_\alpha P_s,\partial_\alpha\Bigl(e^{2V}(DP_s+\calB[P_s]+e^{-2\varphi-2V}\I)\nabla^\perp\psi^{i+1}+e^{-\varphi}P^\perp_s\Bigr)}\,dx\\
&\leq\int\inn{\partial_\alpha P_s,\partial_\alpha\Bigl(e^{2V}(DP_s+\calB[P_s]+e^{-2\varphi-2V}\I)\nabla^\perp\psi^{i+1}\Bigr)}\,dx+C_{\varphi}\norm{\partial_\alpha P_s}_{L^2(\Om)}\norm{P_s}_{H^{\abs{\alpha}}(\Om)}.
\end{align*}
To estimate the remaining term, we argue by interpolation: subtracting and adding the term
\[
R:=\int\inn{\partial_\alpha P_s,e^{2V}\partial_\alpha\Bigl(DP_s+\calB[P_s]+e^{-2\varphi-2V}\I\Bigr)\nabla^\perp\psi^{i+1}}\,dx,
\]
to
\[
\int\inn{\partial_\alpha P_s,\partial_\alpha\Bigl(e^{2V}(DP_s+\calB[P_s]+e^{-2\varphi-2V}\I)\nabla^\perp\psi^{i+1}\Bigr)}\,dx
\]
applying Cauchy-Schwarz and interpolation \eqref{eq:interpol_1}, we obtain that there exists $C_1=C_1(\Om)>0$ such that
\begin{align*}
\int\inn{\partial_\alpha P_s&,\partial_\alpha\Bigl(e^{2V}(DP_s+\calB[P_s]+e^{-2\varphi-2V}\I)\nabla^\perp\psi^{i+1}\Bigr)}\,dx-R+R\\
&\leq C_1\norm{\partial_\alpha P_s}_{L^2(\Om)}\Bigl(\norm{e^{2V}\nabla\psi^{i+1}}_{W^{1,\infty}(\Om)}\norm{DP_s+\calB[P_s]+e^{-2\varphi-2V}\I}_{H^{k-1}(\Om)}\\
&\quad+\norm{DP_s+\calB[P_s]+e^{-2\varphi-2V}\I}_{L^\infty(\Om)}\norm{e^{2V}\nabla\psi^{i+1}}_{H^k(\Om)}\Bigr)+R.
\end{align*}
Taking advantage once again of Lemma \ref{lem:basic_est} and suitable Sobolev embeddings, we just proved that there exists $C_2=C_2(\Om,V,\varphi)>0$ such that
\begin{equation}\label{eq:energy1}
\frac{d}{ds}\frac{1}{2}\int \abs{\partial_\alpha\nabla p^{i+1}_s}^2\,dx \leq C_{2}\norm{\partial_\alpha P_s}_{L^2(\Om)}\bigl(\norm{P_s}_{H^k(\Om)}+1\bigl)\norm{\nabla\psi^{i+1}}_{H^{k}(\Om)}+C_\varphi\norm{\partial_\alpha P_s}_{L^2(\Om)}\norm{P_s}_{H^{k}(\Om)}+R.
\end{equation}
We now estimate the contribution of $R$. First of all, it is easy to control the lower order terms simply by Cauchy-Schwarz and Lemma \ref{lem:basic_est}, obtaining that
\begin{equation}\label{eq:energy2}
\begin{split}
R=\int&\inn{\partial_\alpha P_s,e^{2V}\partial_\alpha\Bigl(DP_s+\calB[P_s]+e^{-2\varphi-2V}\I\Bigr)\nabla^\perp\psi^{i+1}}\,dx\\
&\leq\int\inn{\partial_\alpha P_s,e^{2V}\partial_\alpha(DP_s)\nabla^\perp\psi^{i+1}}\,dx+\norm{\partial_\alpha P_s}_{L^2(\Om)}\norm{e^{2V}\nabla\psi^{i+1}}_{L^\infty}\norm{\calB[P_s]+e^{-2\varphi-2V}\I}_{H^k(\Om)}\\
&\leq \int\inn{\partial_\alpha P_s,e^{2V}\partial_\alpha(DP_s)\nabla^\perp\psi^{i+1}}\,dx+C_3\norm{\partial_\alpha P_s}_{L^2(\Om)}\norm{\nabla\psi^{i+1}}_{H^k(\Om)}(\norm{P_s}_{H^k(\Om)}+1),
\end{split}
\end{equation}
for some constant $C_3=C_3(\Om,V,\varphi)>0$. Finally we get rid of the higher order term integrating by parts:
\begin{equation}\label{eq:energy3}
\begin{split}
\int\inn{&\partial_\alpha P_s,e^{2V}\partial_\alpha(DP_s)\nabla^\perp\psi^{i+1}}\,dx=\int\inn{\nabla\Bigl(\frac{\abs{\partial_\alpha P_s}^2}{2}\Bigr),e^{2V}\nabla^\perp\psi^{i+1}}\,dx
\\&=\int\diva\Bigl(e^{2V}\nabla^\perp\psi^{i+1}\frac{\abs{\partial_\alpha P_s}^2}{2}\Bigr)-\frac{\abs{\partial_\alpha P_s}^2}{2}\diva(e^{2V}\nabla^\perp\psi^{i+1})\,dx\\
&=\int_{\partial\Om}e^{2V}\frac{\abs{\partial_\alpha P_s}^2}{2}\nabla^\perp\psi^{i+1}\cdot\nu\,dx-\int\frac{\abs{\partial_\alpha P_s}^2}{2}\Bigl(e^{2V}\diva(\nabla^\perp\psi^{i+1})+e^{2V}\inn{2\nabla V,\nabla^\perp\psi^{i+1}}\Bigr)\,dx\\
&\leq C_4\norm{\partial_\alpha P_s}_{L^2(\Om)}^2\norm{e^{2V}\nabla\psi^{i+1}}_{L^\infty(\Om)}.
\end{split}
\end{equation}
Combining \eqref{eq:energy1}, \eqref{eq:energy2}, \eqref{eq:energy3}, and summing over $\abs{\alpha}=0,\dots,k$ we obtain the desired result.
\end{proof}
Now that we have obtained  a growth estimate on $\mu^{i+1}_s(x)$ and $\norm{\nabla p^{i+1}_s}_{H^k(\Om)}$, we need analyse the behaviour of the velocity vector field. This last estimate is a direct consequence of the explicit regularity results of Section \ref{sec:elliptic_estimates}.
\begin{proposition}[Elliptic estimates on the velocity]For every $k\geq 4$, there exists some constant $C=C(k,\Om,V\varphi)>0$ such that
\begin{equation}\label{eq:U_EST}
\norm{\nabla\psi^{i+1}}_{H^k(\Om)}\leq C(1-\mu_0^{i+1})^{-(k+1)k-1}\bigr(\norm{\nabla p^{i+1}_0}_{H^{k}(\Om)}+1\bigr)^{(k+1)k}\norm{\nabla p^{i+1}_0}_{H^k(\Om)}.
\end{equation}
\begin{proof}
It suffices to combine Proposition \ref{prop:rescaled_global_estimates} and Lemma \ref{lem:basic_est}, recalling that in our case $\calb$ comes from a rotated gradient by construction.
\end{proof}
\end{proposition}
Combining the estimates on the pressure gradient \eqref{eq:P_EST} and on the velocity vector field \eqref{eq:U_EST} we have that
\begin{equation}\label{eq:combination_P_U}
\frac{d}{ds}\norm{\nabla p^{i+1}_s}_{H^k(\Om)}\leq C(\norm{\nabla p^{i+1}_s}_{H^k(\Om)}+1)\frac{(\norm{\nabla p^{i+1}_0}_{H^k(\Om)}+1)^{k(k+1)}}{(1-\mu^{i+1}_0)^{(k+1)k+1}}\norm{\nabla p^{i+1}_0}_{H^k(\Om)}+C\norm{\nabla p^{i+1}_s}_{H^k(\Om)},
\end{equation}
and similarly by the estimate \eqref{eq:LAMBDA_EST} on $1-\mu_s^{i+1}$ it holds that
\begin{equation}\label{eq:combination_LAMBDA_U}
\frac{d}{ds}(1-\mu^{i+1}_s)\geq -C(\norm{\nabla p^{i+1}_s}_{H^k(\Om)}+1)\frac{(\norm{\nabla p^{i+1}_0}_{H^k(\Om)}+1)^{k(k+1)}}{(1-\mu^{i+1}_0)^{(k+1)k+1}}\norm{\nabla p^{i+1}_0}_{H^k(\Om)}-C\norm{\nabla p^{i+1}_s}_{H^k(\Om)},
\end{equation}
a.e. in $(0,\tau)$. Define
\[
\Theta^{i+1}_s:=\Biggl(\frac{\norm{\nabla p^{i+1}_s}_{H^k(\Om)}+1}{1-\mu^{i+1}_s}\Biggr)^{k(k+1)+2},
\]
together with the monotonically increasing Lipschitz function
\[
\tilde\Theta^{i+1}_s:=\max_{t\in[0,s]}\Theta^{i+1}_t.
\]
The next lemma will constitute the crucial step in the proof of the main Theorem.
\begin{lemma}\label{lem:THETA_dynamic}
There  exits $C=C(\Om,V,\varphi,k)>0$ such that 
\begin{equation}\label{eq:THETA_dynamics}
\frac{d}{ds}\tilde\Theta_s^{i+1}\leq C(\tilde\Theta_s^{i+1})^2,
\end{equation}
almost everywhere in $[0,\tau]$.
\end{lemma}
\begin{proof}In this proof we omit the $(i+1)$ index in our notation. Also, set $M=M(k):=k(k+1)+2$. First of all, by Sobolev embedding we have that there exists $C_1=C_1(\Om)>0$ such that
\[
1\leq\frac{\norm{\I+\calQ[\mol D^2 p_s,\mol\nabla p_s]}_{L^\infty(\Om)}}{1-\mu_s}\leq\frac{1+C_1\norm{\nabla p_s}_{H^3(\Om)}}{1-\mu_s},
\]
hence, up to multiplying all the following estimates by $\max\{1,C_1\}$, we can suppose without loss of generality that
\[
1\leq\frac{1+\norm{\nabla p_s}_{H^k(\Om)}}{1-\mu_s},
\]
for all $s\in[0,\tau]$. In particular we have that
\begin{align*}
\frac{d}{ds}\norm{\nabla p_s}_{H^k(\Om)}&\leq C\max_{t\in[0,s]}\Biggl\{\Biggl(\Bigl(\frac{\norm{\nabla p_t}_{H^k(\Om)}+1}{1-\mu_t}\Bigr)^{k(k+1)+1}+1\Biggl)\norm{\nabla p_t}_{H^k(\Om)}\Biggr\}\\
&\leq 2C\max_{t\in[0,s]}\Biggl\{\Bigl(\frac{\norm{\nabla p_t}_{H^k(\Om)}+1}{1-\mu_t}\Bigr)^{k(k+1)+1}\norm{\nabla p_t}_{H^k(\Om)}\Biggr\}.
\end{align*}
The same bound clearly holds also for $\mu_s$.
Therefore
\begin{align*}
\frac{d}{ds}\Theta_s&= M\Bigl(\frac{\norm{\nabla p_s}_{H^k(\Om)}+1}{1-\mu_s}\Bigr)^{M-1}\Bigl(\frac{1}{1-\mu_s}+\frac{\norm{\nabla p_s}_{H^k(\Om)}+1}{(1-\mu_s)^2}\Bigr)\frac{d}{ds}\norm{\nabla p_s}_{H^k(\Om)}\\
&\leq 2CM\max_{t\in[0,s]}\Biggl\{\Bigl(\frac{\norm{\nabla p_t}_{H^k(\Om)}+1}{1-\mu_t}\Bigr)^{M+k(k+1)}\frac{\norm{\nabla p_t}_{H^k(\Om)}}{1-\mu_t}\Bigl(1+\frac{\norm{\nabla p_t}_{H^k(\Om)}+1}{1-\mu_t}\Bigr)\Biggr\}\\
&\leq 4CM\max_{t\in[0,s]}\Biggl\{\Bigl(\frac{\norm{\nabla p_t}_{H^k(\Om)}+1}{1-\mu_t}\Bigr)^{M+k(k+1)+2}\Biggr\}.
\end{align*}
This proves that
\[
\frac{d}{ds}\Theta_s\leq 4CM\tilde\Theta_s^2.
\]
We distinguish two cases: if  $\frac{d}{ds}\Theta_s\leq 0$, then clearly
\[
\frac{d}{ds}\tilde\Theta_s=0,
\]
and we are done. Otherwise
\[
\frac{d}{ds}\tilde\Theta_s=\frac{d}{ds}\Theta_s\leq 4CM\tilde\Theta^2_s,
\]
completing the proof of the Lemma.
\end{proof}
We only need the following little observation before proving the main result of this section.
\begin{lemma}\label{lem:recursive_relation} Let $\seq{x_i}_{i\geq 0}$ be any real sequence that satisfies for some $c>0$ the recursive relation
\[
x_{i+1}\leq\frac{x_i}{1-cx_i}.
\]
If there exists $N\in\N$ such that $x_0\leq \frac{1}{cN}$, then
\begin{equation}\label{eq:recursive_relation_lemma}
x_{i+1}\leq\frac{x_0}{1-c(i+1)x_0},\,\text{ for every }i\in\{-1,\dots,N-1\}.
\end{equation}
\end{lemma}
\begin{proof}
The statement clearly holds for $i=-1$. Suppose \eqref{eq:recursive_relation_lemma} holds for $0\leq i<N$. Since for every $C>0$ the map $x\mapsto \frac{x}{1-Cx}$ is monotonically increasing and continuous in $(-\infty,\frac{1}{C})$, we have in particular that
\[
x_i\leq\frac{x_0}{1-cix_0}\leq\frac{1}{cN}\frac{N}{N-i}=\frac{1}{c(N-i)}<\frac{1}{c},
\]
and therefore
\[
x_{i+1}\leq\frac{x_i}{1-cx_i}\leq\frac{x_0}{1-cix_0}\frac{1-cix_0}{1-c(i+1)x_0}=\frac{x_0}{1-c(i+1)x_0},
\]
completing the induction.
\end{proof}
We are now ready to prove uniform local-in-time existence for Systems \eqref{eq:discrete_p} and \eqref{eq:discrete_u}. To simplify the statement, we glue together the piecewise approximated solution, naturally defining
\[
\nabla p^{\tau,\epsilon}_t:=\nabla p^i_s,\,\text{ if }t=i\tau+s,
\]
and
\[
\nabla \psi^{\tau,\epsilon}_t:=\nabla\psi^i,\,\text{ if }t\in[i\tau,(i+1)\tau).
\]
\begin{theorem}\label{thm:discrete_existence}Let $\Om$ be a bounded subset of $\R^2$ with boundary fo class $C^{k+1}$ and $\nabla p_0\in H^k(\Om,\R^2)$ such that
\[
\calQ[D^2 p_0,\nabla p_0]\geq-\mu_0\I>-\I,
\]
for some $\mu_0<1$ and $k\geq 4$. Then, there exists a constant $C=C(\Om,V,\varphi,k)>0$ and $t^*>0$ such that for every $\tau=t^*/N$, $N\in\N$ big enough and $\epsilon>0$, there exists an approximate solution  
\[
\{\nabla p^i_s,\nabla \psi^i\}_{i=0}^{N-1}\in C^1([0,\tau],H^k(\Om,\R^2))\times H^k(\Om,\R^2),
\]
of Systems \eqref{eq:discrete_p} and \eqref{eq:discrete_u}, where $t^*$ can be taken equal to
\[
t^*:=C\Biggl(\frac{1-\mu_0}{\norm{\nabla p_0}_{H^k(\Om)}+1}\Biggr)^{(k+1)k+2}.
\]
In particular, for every $0<t'<t^*$, there exists $C'=C'(\Om,V,\varphi,k)>0$ such that
\[
\norm{\nabla p^{\tau,\epsilon}_t}_{H^k(\Om)},\norm{\nabla\psi^{\tau,\epsilon}_t}_{H^k(\Om)}\leq C',
\]
for all $t\in[0,t']$.
\end{theorem}
\begin{proof}Integrating for $s\in[0,\tau]$ Equation\eqref{eq:THETA_dynamics} of Lemma \ref{lem:THETA_dynamic} at time $i$, and recalling that $\tilde\Theta^{i+1}_0=\tilde\Theta^i_\tau$, we obtain the recursive relation
\[
\tilde\Theta_0^{i+1}\leq\frac{\tilde\Theta^i_0}{1-C\tau\tilde\Theta_0^i},
\]
which, applying Lemma \ref{lem:recursive_relation} gives the bound
\begin{equation}\label{eq:recursive_2}
\tilde\Theta_0^{i+1}\leq\frac{\tilde\Theta^0_0}{1-C\tau(i+1)\tilde\Theta_0^0},
\end{equation}
for every $i=\{-1,0,1,\dots,N-1\}$ provided
\[
\tilde\Theta^0_0=\Theta_0^0\leq\frac{1}{C\tau N},
\]
for some $N\in\N$. Hence, setting
\[
t^*:=\frac{1}{C\Theta^0_0},
\]
we ensure the local existence of an approximate solution in $[0,t^*)$ uniformly in $\epsilon>0$ and for every $\tau=t^*/N$, $N\in \N$ big enough. In particular, \eqref{eq:recursive_2} implies that for any interval of time $[0,t']$ with $t'<t^*$ the uniform bound
\[
\tilde\Theta_s^i\leq C',
\]
holds, where $C'>0$ can be taken such that
\[
t'=\frac{1}{C}\Bigl(\frac{1}{\Theta_0^0}-\frac{1}{C'}\Bigr).
\]
\end{proof}

\section{Compactness argument and proof of the main Theorem}\label{sec:compactness}

Fix any $0<t'<t^*$, and $N_0\in\N$ large. For every $N\geq N_0$ define
\[
\nabla p^N_t:=\nabla p_t^{t'/2^N,t'/2^N},
\]
and
\[
\nabla\psi^N_t:=\nabla\psi^{t'/2^N,t'/2^N}_t.
\]
Then, by Theorem \ref{thm:discrete_existence}, the sequence $\seq{\nabla p^N_t}_{N\geq N_0}$ is uniformly bounded in the space
\[
\mathcal{W}:=\Bigl\{\nabla q_t\in L^\infty(0,t';H^{k}(\Om,\R^2)),\text{ and }\partial_t\nabla q_t\in L^\infty(0,t';H^1(\Om,\R^2)))\Bigr\}.
\]
Since the embedding of $H^k(\Om)$ in $C^{k-2,\alpha}(\Om)$ is compact (see \cite[Chapter 6]{A08}) and $C^{k-2,\alpha}(\Om)$ embeds continuously in $H^1(\Om)$, by Aubin-Lions-Simons Lemma we have that
\[
\mathcal{W}\hookrightarrow C(0,t';C^{k-2,\alpha}(\Om,\R^2)),
\]
is compact as well. Extracting a converging sub-sequence we obtain (after relabelling) that
\[
\nabla p_t^N\to \nabla p_t\,\text{ in }C(0,t';C^{k-2,\alpha}(\Om,\R^2)),
\]
for  some $\nabla p_t\in C(0,t';C^{k-2,\alpha}(\Om,\R^2))$. Moreover, looking at $\seq{\nabla p_t^N}_{N\geq N_0}$ as bounded subset of the space $L^2(0,t';H^k(\Om,\R^2))$, we can affirm that
\[
\nabla p_t^N\rightharpoonup \nabla p_t\text { in }L^2(0,t';H^k(\Om,\R^2)).
\]
Let $\nabla \psi_t$ be solution of the System \eqref{eq:elliptic_u} associated to the limit $\nabla p_t$, i.e.
\begin{equation}
\begin{cases}
\diva(e^{-2\varphi}(\I+\calQ[D^2p_t,\nabla p_t])\nabla\psi_t)+\nabla^\perp(\calf[\nabla p_t])\cdot\nabla \psi_t=\diva(\calF[\nabla p_t]),&\text{ in }\Om\\
\psi_t=0,&\text{ on }\partial\Om.
\end{cases}
\end{equation}
Observe that the lower bound on the uniform elliptic constants $1-\mu^i_s$ proved in Theorem \ref{thm:discrete_existence} propagates to the limit, that we will denote with
\[
-\mu_t:=\inf\Bigl\{\inn{\calQ[D^2p_t,\nabla p_t](x)\xi,\xi}:\abs{\xi}=1,\,x\in\Om\Bigr\}.
\]
By qualitative elliptic regularity, we can affirm that $\nabla\psi_t\in C(0,t';C^{k-2,\alpha}(\Om,\R^2))$.
Fix $t\in(0,t')$ and let
\[
t_N:=\min\{jt'/2^N\geq t :j=0,\dots,N\},
\]
and observe that the difference $\psi_t-\psi_{t_N}^N=\psi_t-\psi_{t}^N$ solves the equation
\[
\begin{cases}
\diva(e^{-2\varphi}(\I+\calQ[D^2p_t,\nabla p_t])\nabla(\psi_t-\psi_t^N))+\nabla^\perp(\calf[\nabla p_t])\cdot\nabla (\psi_t-\psi_t^N)=\calX^N_t,&\text{ in }\Om\\
\psi_t-\psi_t^N=0,&\text{ on }\partial\Om,
\end{cases}
\]
where
\[
\calX^N_t:=\diva\bigl(\calF[\nabla p_t-\nabla p_t^N]\bigr)-\diva\bigl(e^{-2\varphi}\calQ[D^2 (p_t -p_t^N),\nabla (p_t-p_t^N)]\nabla\psi^N_t\bigr)-\nabla^\perp(\calf[\nabla(p_t-p_t^N)])\cdot\nabla\psi^N_t.
\]
We can argue as at the end of Proposition \ref{prop:global_estimates}, to estimate
\[
\norm{\nabla \psi_t-\nabla\psi^N_t}_{L^2(\Om)}\leq \frac{1}{1-\mu_t}\norm{\calX^N_t}_{L^2(\Om)}\to 0,
\]
uniformly in $(0,t')$ thanks to the bounds given by Theorem \ref{thm:discrete_existence}. Moreover, by weak compactness of $L^2(0,t';H^k(\Om,\R^2))$, we have that
\[
\nabla\psi^N_t\rightharpoonup\nabla\psi_t\in L^2(0,t';H^k(\Om,\R^2)).
\]
To summarise, we have the following proposition.
\begin{proposition}\label{prop:compactness}Up to taking a subsequence of $\seq{\nabla p_t^N,\nabla\psi^N_t}_{N\geq N_0}$ there exist
\[
\nabla p_t,\nabla\psi_t\in C(0,t';C^{k-2,\alpha}(\Om,\R^2))\cap L^2(0,t';H^k(\Om,\R^2)),
\]
such that
\[
\nabla p_t^N\to\nabla p_t,
\]
strongly in $C(0,t';C^{k-2,\alpha}(\Om,\R^2))$ and weakly in $L^2(0,t';H^k(\Om,\R^2))$, and
\[
\nabla \psi^N_t\to\nabla\psi_t,
\]
strongly in $L^\infty(0,t';C^{k-2,\alpha}(\Om,\R^2))$ and weakly in $L^2(0,t';H^k(\Om,\R^2))$.
\end{proposition}
We are now ready to prove the main result of this paper.
\begin{proof}[Proof of Theorem \ref{thm_MAIN}] We have to show that our candidates $(\nabla p_t,\nabla\psi_t)$ form a solution of the semigeostrophic equation \eqref{eq:SGflat}. We first prove that $(\nabla p_t,\nabla\psi_t)$ is a weak solution, the conclusion follows from the additional regularity showed before. Let $\xi_t\in C^1_c([0,t'),C^\infty(\Om,\R^2))$ be any test function, denote with $\innn{\cdot,\cdot}$  the standard inner product of $L^2(0,t';L^2(\Om,\R^2))$ and with $\molN:=\mathfrak{I}_{j'/N}$. Then, testing \eqref{eq:discrete_p} against $\xi_t$ we have that
\[
0=\innn{\nabla p_t^N,\partial_t\xi_t}-\innn{\nabla p_0,\xi_0}+\innn{\hel\molN\xi_t,e^{2V}\Bigl(\molN D^2p_t^N+\calB[\molN\nabla p^N_t]+e^{-2\varphi-2V}\I\Bigr)\nabla^\perp\psi^N_t+e^{-\varphi}\molN\nabla^\perp p^N_t}.
\]
Then, we can write
\begin{align*}
\innn{\nabla p_t,\partial_t\xi_t}&-\innn{\nabla p_0,\xi_0}+\innn{\hel\xi_t,e^{2V}\Bigl(D^2p_t+\calB[\nabla p_t]+e^{-2\varphi-2V}\I\Bigr)\nabla^\perp\psi_t+e^{-\varphi}\nabla^\perp p_t}=\\
&\innn{\nabla(p_t-p_t^N),\partial_t\xi_t}+\innn{\hel\xi_t-\molN\hel\xi_t,e^{2V}\Bigl(D^2p_t+\calB[\nabla p_t]+e^{-2\varphi-2V}\I\Bigr)\nabla^\perp\psi_t+e^{-\varphi}\nabla^\perp p_t}\\
&\quad+\innn{\hel\molN\xi_t,e^{2V}\Bigl(D^2(p_t-\molN p^N_t)+\calB[\nabla (p_t-\molN p_t)]\Bigr)\nabla^\perp\psi_t+e^{-\varphi}\nabla^\perp (p_t-\molN p_t^N)}\\
&\quad+\innn{\hel\molN\xi_t,e^{2V}\Bigl(\molN D^2p_t^N+\calB[\molN\nabla p^N_t]+e^{-2\varphi-2V}\I\Bigr)\nabla^\perp(\psi^N_t-\psi_t)},
\end{align*}
which goes to zero as $N$ goes to $+\infty$, thanks to the uniforms bounds of Theorem \ref{thm:discrete_existence} and Proposition \ref{prop:compactness}. Therefore, we have that $(\nabla p_t,\nabla\psi_t)$ solves weakly
\[
\partial_t\nabla p_t=\hel\Bigl(e^{2V}(D^2p_t+\calB[\nabla p_t]+e^{-2\varphi-2V}\I)\nabla^\perp\psi_t+e^{-\varphi}\nabla^\perp p_t\Bigr)=:\hel(X_t).
\]
We now take advantage of the elliptic equation solved by $\psi_t$ in order to get rid of the Hodge-Helmholtz decomposition in the right-hand side. Here is  the only point in the proof where we need to assume $\Om$ simply connected (see Remark \ref{rmk:T^2} for the periodic case $\Om=\R^2/\Z^2=\T^2$). The orthogonal complementary of $\hel(\cdot)$
\[
w_t:=X_t-\hel(X_t),
\]
is tangent to $\partial\Om$ and divergence free by construction of $\hel(X_t)$. Moreover, since
\[
\curl(w_t)=-\diva(X^\perp)=-\diva\bigl(e^{-2\varphi}(\I+\calQ[D^2p_t,\nabla p_t])\nabla\psi_t+\nabla^\perp(\calf[\nabla p_t])\cdot\nabla\psi-\calF[\nabla p_t]\bigr)=0,
\]
by construction of $\nabla\psi_t$, we conclude that $w_t$ in an harmonic vector field, and hence equal to zero since $\Om$ is simply connected. Therefore, $X_t=\hel(X_t)$.
\end{proof}

\begin{remark}\label{rmk:T^2}With some minor adjustments, it is possible to include the not simply connected flat periodic case $\Om=\T^2=\R^2/\Z^2$, $V=\varphi=0$. We have to substitute in Equation \eqref{eq:elliptic_u} the boundary condition $\psi=0$ on $\partial\Om$ with $\int_{\T^2}\psi\,dx=0$, and impose periodicity conditions on $\psi^{i+1}_s$, $p_0$ and $p^{i+1}_s$. We need also to adjust the operator $\hel(X)=\nabla q$, defined now to be the inverse operator of the problem
\[
\begin{cases}
\Delta q=\diva(X),\\
\int_{\T^2}q\,dx=0.
\end{cases}
\]
Existence of an uniform regularized solution that converges on $[0,t']$ to $\seq{\nabla p_t,\nabla\psi_t}$ still holds. The only problem to fix is that there exist non-trivial harmonic fields on $\T^2$. However, they do not play any role in our problem, and this can be showed with a direct computation: recall that we are in the situation
\[
\partial_t\nabla p_t=\hel\Bigl((D^2p_t+\I)\cdot\nabla^\perp\psi_t+\nabla^\perp p_t\Bigr)=\hel(X_t),
\]
and we want to get rid of $\hel$. Write
\[
X_t=\hel(X_t)+w_t=\nabla q_t+w_t,
\]
where $\diva(w_t)=\curl(w_t)=0$ thanks to the construction of $\nabla\psi_t$ and $\hel(X)$. Therefore, by duality we can see $w_t$ as an element of the de Rahm Cohomology $H^1_{dR}(\T^2)\cong\R^2$, which is generated by the two covector fields $dx^1$ and $dx^2$, which are closed but not exact since $x\mapsto x^1$ and $x\mapsto x^2$ are not periodic functions. Hence, there exist $\alpha^1_t,\alpha^2_t\in \R$ such that
\[
w_t=\alpha_t^1\frac{\partial}{\partial x^1}+\alpha_t^2\frac{\partial}{\partial x^2}.
\]
Now, choose $k\in\{1,2\}$, and observe that
\[
\int_{\T^2}\inn{X_t,\frac{\partial}{\partial x^k}}\,dx=\int_{\T^2}\inn{\nabla q_t+w_t,\frac{\partial}{\partial x^k}}\,dx=\alpha^k_t,
\]
Hence, taking advantage of the explicit form of $X_t$ and integrating by parts we conclude that
\begin{align*}
\alpha^k_t&=\int_{\T^2}\inn{(D^2p_t+\I)\cdot\nabla^\perp\psi_t+\nabla^\perp p_t,\frac{\partial}{\partial x^k}}\,dx=\int_{\T^2}\inn{D^2p_t\cdot\nabla^\perp\psi_t,\frac{\partial}{\partial x^k}}\,dx=\int_{\T^2}\inn{\nabla(\partial_k p_t),\nabla^\perp\psi_t}\,dx\\
&=\int_{\T^2}\diva(\partial_k p_t\nabla^\perp\psi_t)+\partial_k p_t\cdot\diva(\nabla^\perp\psi_t)\,dx=0.
\end{align*}
This shows $\hel(X_t)=X_t$ as wished.
\end{remark}

\bigskip
\printbibliography

@misc{ACPF12,
      title={Existence of Eulerian solutions to the semigeostrophic equations in physical space: the 2-dimensional periodic case}, 
      author={Luigi Ambrosio and Maria Colombo and Guido De Philippis and Alessio Figalli},
      year={2012},
      eprint={1111.7202},
      archivePrefix={arXiv},
      primaryClass={math.AP}
}

@article{BB71,
 ISSN = {00361399},
 URL = {http://www.jstor.org/stable/118357},
 author = {J.-D. Benamou and Y. Brenier},
 journal = {SIAM Journal on Applied Mathematics},
 number = {5},
 pages = {1450--1461},
 publisher = {Society for Industrial and Applied Mathematics},
 title = {Weak Existence for the Semigeostrophic Equations Formulated as a Coupled Monge-Ampère/Transport Problem},
 volume = {58},
 year = {1998}
}

@book{MB01, place={Cambridge}, series={Cambridge Texts in Applied Mathematics}, title={Vorticity and Incompressible Flow}, DOI={10.1017/CBO9780511613203}, publisher={Cambridge University Press}, author={Majda, Andrew J. and Bertozzi, Andrea L.}, year={2001}, collection={Cambridge Texts in Applied Mathematics}}

@book{C06,
author = {Cullen, Michael J.P.},
copyright = {LOC 20111024 Droits réservés},
isbn = {186094518X},
language = {eng},
address = {London},
publisher = {Imperial College Press},
title = {A mathematical theory of large-scale atmosphere/ocean flow},
year = {2006},
}

@article{FP12,
author = {Philippis, Guido and Figalli, Alessio},
year = {2012},
month = {02},
pages = {},
title = {Second order stability for the Monge-Ampere equation and strong Sobolev
convergence of optimal transport maps},
volume = {6},
journal = {Analysis and PDE},
doi = {10.2140/apde.2013.6.993}
}

@book{E15,
author = {Evans, Lawrence C.},
edition = {2nd ed., repr. with corr.},
isbn = {9780821849743},
language = {eng},
address = {Providence, R.I},
publisher = {American Mathematical Society},
series = {Graduate studies in mathematics vol. 19, ed. 2, repr.},
title = {Partial differential equations},
year = {2015},
}

@article {H75,
      author = "Brian J.  Hoskins",
      title = "The Geostrophic Momentum Approximation and the Semi-Geostrophic Equations",
      journal = "Journal of Atmospheric Sciences",
      year = "1975",
      publisher = "American Meteorological Society",
      address = "Boston MA, USA",
      volume = "32",
      number = "2",
      pages=      "233 - 242"
}

@article{L06,
author = {Loeper, Grégoire},
year = {2006},
month = {01},
pages = {795-823},
title = {A Fully Nonlinear Version of the Incompressible Euler Equations: The Semigeostrophic System},
volume = {38},
journal = {SIAM J. Math. Analysis},
}

@book{A08,
author = {Adams, Robert A.},
edition = {2nd ed., [Repr.]},
isbn = {9780120441433},
language = {eng},
address = {Amsterdam [etc},
publisher = {Academic Press},
series = {Pure and applied mathematics series 140, repr. 2008},
title = {Sobolev spaces},
year = {2008},
}

@book{A18,
author = {Ambrosio, Luigi},
isbn = {9788876426506},
language = {eng},
address = {Pisa},
publisher = {Edizioni della Normale},
series = {Appunti Lecture Notes 18},
title = {Lectures on elliptic partial differential equations},
year = {2018},
}

\end{document}